\documentclass[11pt]{amsart}

\usepackage{amsfonts,amsmath,latexsym,amssymb,verbatim,amsbsy,amsthm}

\usepackage{graphicx}

\usepackage[top=1in, bottom=1in, left=1in, right=1in]{geometry}

\usepackage[dvipsnames]{xcolor}

\usepackage[colorlinks=true, pdfstartview=FitV, linkcolor=RoyalBlue,citecolor=ForestGreen, urlcolor=blue]{hyperref}

\theoremstyle{plain}
\newtheorem{THEOREM}{Theorem}[section]

\newtheorem{theorem}[THEOREM]{Theorem}
\newtheorem{corollary}[THEOREM]{Corollary}
\newtheorem{lemma}[THEOREM]{Lemma}

\theoremstyle{definition}

\theoremstyle{remark}

\newtheorem{remark}[THEOREM]{Remark}

\numberwithin{equation}{section}
\numberwithin{figure}{section}

\newcommand{\thm}[1]{Theorem~\ref{#1}}

\def \a {\alpha}
\def \b {\beta}
\def \g {\gamma}
\def \d {\delta}
\def \e {\mu} %far-range intercations {\varepsilon}

\def \l {\lambda}
\def \n {\nabla}

\def \L {\Lambda}

\def \O {\Omega}
\def \calC {{\mathcal C}}

\def \ba {{\bf a}}

\def \bu {{\bf u}}
\def \bv {{\bf v}}
\def \bx {{\bf x}}
\def \by {{\bf y}}
\def \bw {{\bf w}}

\def \bF {{\bf F}}

\def \bQ {{\bf Q}}

\def \bX {{\bf X}}
\def \bY {{\bf Y}}
\def \bV {{\bf V}}

\def \cA {\mathcal{A}}

\def \cD {\mathcal{D}}
\def \cE {\mathcal{E}}

\def \cI {\mathcal{I}}

\def \cK {\mathcal{K}}

\def \cO {\mathcal{O}}
\def \cP {\mathcal{P}}

\def \cX {\mathcal{X}}

\newcommand{\R}{\ensuremath{\mathbb{R}}}   %%% reals
\newcommand{\T}{\ensuremath{\mathbb{T}}}   %%% torus

\def \loc {\mathrm{loc}}

\def \lan {\langle}
\def \ran {\rangle}

\def \p {\partial}

\def \ss {\subset}

\renewcommand{\geq}{\geqslant}
\renewcommand{\ge}{\geqslant}
\renewcommand{\leq}{\leqslant}

\newcommand{\radial}[1]{|#1|}

\DeclareMathOperator{\supp}{supp} %
\DeclareMathOperator{\conv}{conv} %
\DeclareMathOperator{\diam}{diam} %
\def \dbx  {\, \mbox{\upshape{d}}\bx}
\def \dby  {\, \mbox{\upshape{d}}\by}

\def \dy  {\, \mbox{d}y}

\def \dr  {\, \mbox{d}r}
\def \ds  {\, \mbox{d}s}

\def \dbv  {\, \mbox{d}\bv}

\def \ddt  {\frac{\mbox{d\,\,}}{\mbox{d}t}}

\def \Exp {\mathrm{Exp}}

\begin{document}

\title[Multi-flocks: emergent dynamics in multi-scale collective behavior]{Multi-flocks: emergent dynamics  in\\ systems with multi-scale collective behavior}

\author{Roman Shvydkoy}
\address{Department of Mathematics, Statistics, and Computer Science, 
	University of Illinois, Chicago}
\email{shvydkoy@uic.edu}

\author{Eitan Tadmor}
\address{Department of Mathematics, Center for Scientific Computation and Mathematical Modeling (CSCAMM), and Institute for Physical Sciences \& Technology (IPST), University of Maryland, College Park}
\email{tadmor@cscamm.umd.edu}

\date{\today}



\thanks{\textbf{Acknowledgment.} Research was supported in part by NSF grants 
	DMS16-13911, RNMS11-07444 (KI-Net) and ONR grant 
	N00014-1812465 (ET), and by NSF
	grant  DMS-1813351 (RS). RS thanks CSCAMM and University of Maryland for hospitality during the preparation of this paper.
}

\begin{abstract}
We  study  the multi-scale description of large-time  collective behavior
of agents driven by alignment. 
The resulting \emph{multi-flock dynamics}  arises naturally with realistic  initial configurations   consisting of multiple spatial scaling, which in turn peak at different time scales. We derive a `master-equation' which describes a complex multi-flock congregations governed by two ingredients: (i) a fast inner-flock communication; and  (ii) a  slow(-er) inter-flock communication. The latter is driven by macroscopic observables which feature the \emph{up-scaling} of the problem. We extend the current mono-flock theory, proving a series of results which describe rates of  multi-flocking  with natural dependencies on communication strengths. Both agent-based, kinetic, and hydrodynamic descriptions are considered, with particular emphasis placed on the discrete and macroscopic descriptions. 
\end{abstract}

\maketitle
\setcounter{tocdepth}{1}

\centerline{\dedicatory{\emph{\large To Bj\"{o}rn Engquist with friendship and appreciation}}}

\bigskip
\tableofcontents
\section{Introduction}

We present (to our knowledge --- a first) systematic study  of multi-scale analysis for 
the large-time  behavior of collective dynamics.
Different scales of the dynamics are captured by different descriptions. Our starting point is an agent-based description of \emph{alignment dynamics} in which a crowd of $N$ agents, each with unit mass,  identified by (position, velocity) pairs  $(\bx_i(t), \bv_i(t)) \in \R^d\times \R^d$, are governed by
\begin{equation}\label{eq:ial}
\begin{split}
\dot{\bx}_i(t)&=\bv_i(t)\\
\dot{\bv}_i(t) &= \lambda\sum_{j\in {\calC}} \phi(\bx_i,\bx_j)(\bv_j(t)-\bv_i(t)), \qquad i\in {\mathcal C}:=\{1,2,\ldots, N\}.
\end{split}
\end{equation}
The alignment dynamics is  dictated by the \emph{symmetric} communication kernel $ \phi(\cdot,\cdot)\geq 0$.
It is tacitly assumed here that the initial configuration of the agents are  equi-distributed which justifies a scaling  factor $\lambda=1/N$, and  thus \eqref{eq:ial} amounts to the celebrated Cucker-Smale (CS) model \cite{CS2007a,CS2007b}. The tendency   to align  velocities leads to the generic large-time formation of a flock.\newline
In realistic scenarios, however, initial configurations are not equi-distributed. Indeed, fluctuations in initial density may admit different scales of spatial concentrations. What is the  collective behavior subject to such non-uniform initial densities? This is the main focus of our work. 

The presence of  different spatial scales leads to formation of separate flocks at different time scales, which are realized  by mixing  different formulations of alignment dynamics --- from agent-based to hydrodynamic descriptions. 
In section \ref{sec:multi} we make a systematic derivation, starting with the agent-based CS dynamics for a single flock  \eqref{eq:ial} and ending with  dynamics which involves several flocks ${\calC}_\a, \a=1, \ldots A$: the  $\a$-flock consists of $N_\a$ agents, identified by (position, velocity) pairs $\{(\bx_{\a i},\bv_{\a i})\}_{i\in {\calC}_a}$, which is one part of a total crowd  of size $N=\sum_{i=1}^A N_\a$. The resulting \emph{multi-flock} dynamics is governed by a `master-equation'
\begin{equation}\label{e:CS}
\left\{
\begin{split}
\dot{\bx}_{\a i}&=\bv_{\a i},\\
\dot{\bv}_{\a i}& = \l_\a  \sum_{j = 1}^{N_\a} m_{\a j}  \phi_\a(\bx_{\a i}, \bx_{\a j}) (\bv_{\a j}-\bv_{\a i}) + \e \sum_{\substack{\b=1 \\ \b \neq \a}}^{A} M_\b\psi(\bX_\a,\bX_\b) (\bV_{\b}-\bv_{\a i}).
\end{split}\right. 
\end{equation}

The system \eqref{e:CS}  arises naturally as an effective description for the alignment dynamics  with multiple spatial scaling, which in turn, yields multiple \emph{temporal} scalings. Such multi-scaling appears when each $\a$-flock undergoes evolution  on a time scale much shorter than relative evolution between the flocks. 
Accordingly, the dynamics in \eqref{e:CS} has two main parts. The first sum on the right encodes short-range alignment interactions among agents in flock $\a$, dictated by \emph{symmetric} communication kernel $\phi_\a$ with amplitude $\l_\a$. The new feature here is that  spatial variations in initial density require us to trace  the different masses $m_{\a j}$ attached to different agents located at $\bx_{\a j}$. The second sum on the right encodes the interactions between agents in flock $\alpha$ and the `remote' flocks $\beta\neq \alpha$. The communication is  dictated by  symmetric kernel $\psi$ with amplitude $\e$: since these are long-range interactions, they are scaled with  relatively weak amplitude $\e \ll 1$, and we therefore do not get into finer resolution of different kernels, $\psi_{\a\b}$, to different flocks (inter-flocking interactions driven by different $\psi_{\a\b}$  is the topic of a recent study on  \emph{multi-species} dynamics \cite{HeT2019}).  The new feature here is that  the remote flocks in these long-range interactions, $\calC_{\b\neq \a}$, are encoded in terms of their macroscopic `observables' --- their mass, $M_\b = \sum_{i \in \calC_\b} m_{\b i}$, and  centers of mass  and momentum
\[
\bX_\b  := \frac{1}{M_\b} \sum_{i \in \calC_\a} m_{\b i} \bx_{\b i}, \quad \bV_\b := \frac{1}{M_\b} \sum_{i \in \calC_\b} m_{\b i} \bv_{\b i}, \qquad M_\b := \sum_{i \in \calC_\b} m_{\b i}.
\]

\ifx%%%
Denoting by $\l_\a$ the communication strength within the $\a$-flock, and by $\e$ the inter-flock strength, we introduce the following Cucker-Smale type multi-flock system:
To set the notation we consider  several flocks $\{\bx_{\a i}\}_{i=1}^{N_\a}$ indexed by $\a = 1,\ldots,A$, in which communication is governed by an inter-flock kernel $\psi$, and a family of inner-flock kernels $\phi_\a$'s.   Denoting $\dot{\bx}_{\a i} = \bv_{\a i}$ and assigning a ``mass" $m_{\a i}$  to each agent $\bx_{\a i}$  we consider the following macroscopic quantities:
\[
\begin{split}
\mbox{Mass of $\a$-flock /system:} &  \quad  , \quad M = \sum_{\a=1}^A M_{\a}\\
\mbox{Center of mass of $\a$-flock/system:} &  \quad , \quad \bX = 
\frac{1}{M} \sum_{\a} M_{\a} \bX_{\a} \\
\mbox{Total momentum of $\a$-flock/system:} &  \quad , \quad \bV = 
\frac{1}{M} \sum_{\a} M_{\a} \bV_{\a}.
\end{split}
\]
\fi%%%

These macroscopic quantities $\{(\bX_\a,\bV_\a)\}$ are  determined  by the slow inetr-flocking dynamics:  a weighted sum 
$\sum_i m_{\a i}$\eqref{e:CS}${}_{i}$ yields   
\begin{equation}\label{e:CSmacro}
\left\{
\begin{split}
\dot{\bX}_{\a }&=\bV_{\a},\\
\dot{\bV}_{\a}& =\e \sum_{\b \neq \a} M_\b \psi(\bX_\a,\bX_\b)(\bV_{\b}-\bV_{\a}).
\end{split}\right. 
\end{equation}
Thus, starting with agent dynamics \eqref{eq:ial} we end up with the same classical Cucker-Smale dynamics \eqref{e:CSmacro} for `super-agents', weighted by their respective masses and representing macroscopic parameters of those flocks.
\begin{remark}({\bf Smooth and singular kernels}). In the case when the inter-flock and internal communication kernels are smooth, the global existence of the system \eqref{e:CS} follows by a trivial application of the Picard iteration  and continuation. If the kernels $\phi_\a$ are singular, however,  collisions lead to finite time blowup, so this case needs to be addressed separately. In the Appendix  we show that multi-flock dynamics 
 governed by  singular communication kernels with `fat-head' so that
$\int_0^1 \phi_\a(r) \dr = \infty$, experiences no internal collisions.
Consequently, one can deduce global existence for systems with smooth $\psi$ and a family of either smooth kernels or `fat-head' kernels. 
\end{remark}
  %%%%%%%%%%%%%%%%%%%
 \subsection{Statement of main results}
 %%%%%%%%%%%%%%%%%%%%%%
Much of the  theory available in the literature on mono-scale flocking, e.g., \cite{BDT2017,BDT2019} and the references therein, admits proper extension  to the framework of multi-flocks. We chose to carry out proofs to three main aspects of (i) the large-time \emph{alignment} behavior of \eqref{e:CS}; (ii) multi-flocks in presence of additional  \emph{attractive forcing};  and (iii) large-crowd \emph{hydrodynamics} of multi-flocks. Below we highlight  the main results.
 
 We begin, in section \ref{sec:rate}, with the large-time alignment behavior of the multi-flock dynamics \eqref{e:CS}. We assume that the short- and long-range communication kernels $\phi_\a$ and $\psi$ are \emph{bounded} and \emph{fat-tailed} in the sense that\footnote{Here and below we abbreviate $\langle X \rangle := (1+|X|^2)^{1/2}$} 
 \begin{equation}\label{eq:ifattail} 
 \phi_\a(\bx,\by)\gtrsim \langle|\bx-\by|\rangle^{-\eta_\a}, \quad \psi(\bx,\by)\gtrsim \langle|\bx-\by|\rangle^{-\zeta}, \qquad \eta_a,\zeta\leq 1.
 \end{equation}
 They dictate the fast  alignment rates insides flocks and slow cross-flocks rates, summarized in the following two theorems.
 \begin{theorem}[{\bf Fast local flocking}]\label{t:fast}
Assume that the communication in an $\a$-flock has a fat-tailed kernel $\phi_\a(\bx,\by)\gtrsim \langle|\bx-\by|\rangle^{-\eta_\a}, \ \eta_\a\leq 1$. Then,  the diameter of the $\a$-flock
is uniformly bounded in time,  $\cD_\a(t) :=\max_{i,j}|\bx_{\a i}(t) - \bx_{\a j}(t)|\leq \overline{\cD}_{\a}$, and the $\a$-flock aligns  exponentially fast towards its center of momentum
\begin{equation}\label{eq:fast}
\max_{i}| \bv_{\a i}(t) - \bV_\a(t)| \lesssim e^{-\d_\a t}, \qquad  \d_\a=\l_\a M_\a (\overline{\cD}_{\a})^{-\eta_\a}.
\end{equation}
\end{theorem}
The main message of this theorem is that the $\a$-flock alignment towards $\bV_\a$   depends only  on the $\a$-flock  own parameters, but not the global values.  
The  global alignment has a slow(-er) rate  reflecting weaker communication due to the smaller amplitude $\e$ and the global diameter of the multi-flock $\cD$.
Let $\bV$ denote the  center of momentum of the whole crowd,
$ \bV:=\frac{1}{M} \sum_{\a} M_{\a} \bV_{\a}(t)$, and observe that it is time invariant.
\begin{theorem}[{\bf Slow global flocking}]\label{t:slow}
Suppose $\psi$ has a fat tail, $\psi(\bx,\by)\gtrsim \langle|\bx-\by|\rangle^{-\zeta}, \zeta\leq 1$. Then the diameter of the whole crowd is uniformly bounded in time, $\cD(t): = \max_{\a,\b}|\bX_{\a}(t) - \bX_{\b}(t)| \leq \overline{\cD}$, and solutions of \eqref{e:CS} globally align with the global center of momentum $\bV$,
\begin{equation}\label{eq:slow}
\max_{\a, i}| \bv_{\a i}(t) - \bV| \lesssim e^{-\d t}, \qquad \d ={\e M}{(\overline{\cD})^{-\zeta}}.
\end{equation}
\end{theorem}
As a consequence of the two theorems above we obtain what is called ``strong flocking", that is when all the displacements between agents stabilize,
%\[
$\bx_{\a i}(t) - \bx_{\a j}(t) \to \bar{\bx}_{\a ij}$  as $t \to \infty$.
%\]
\newline
Indeed, 
$\displaystyle 
\bx_{\a i}(t) - \bx_{\a j}(t)  = \bx_{\a i}(0) - \bx_{\a j}(0) + \int_0^t [ \bv_{\a i}(s) - \bv_{\a j}(s) ] \ds$, hence 
\[
\bar{\bx}_{\a ij} =  \bx_{\a i}(0) - \bx_{\a j}(0) + \int_0^\infty [ \bv_{\a i}(s) - \bv_{\a j}(s) ] \ds,
\]
and the rate of convergence is obviously the same as that claimed for the velocities.

\smallskip
In section \ref{sec:attr} we study the multi-flock dynamics \eqref{e:CS} with additional \emph{attractive forcing} (here we restrict attention to interactions determined by a radially symmetric kernels)
\begin{equation}\label{e:iCSattr}
 \left\{
\begin{split}
\dot{\bx}_{\a i}&=\bv_{\a i},\\
\dot{\bv}_{\a i}& = \frac{1}{N_\a} \sum_{j = 1}^{N_\a}  m_{\a j} \phi_\a(\radial{\bx_{\a i}- \bx_{\a j}}) (\bv_{\a j}-\bv_{\a i}) + \e \sum_{\substack{\b=1 \\ \b \neq \a}}^{A} M_{\b}\psi(\radial{\bX_\a-\bX_\b}) (\bV_{\b}-\bv_{\a i}) + \bF_{\a i}.
\end{split}\right. 
\end{equation}
Here,  $\bF_{\a i}(t)= - \frac{1}{N_\a} \sum_{j=1}^{N_\a} \n U(\radial{\bx_{\a i} - \bx_{\a j}})$ is an external attractive forcing   induced by a convex potential  $U$ which belongs to the class of potentials outlined in \eqref{e:Uattr} below.  Arguing along the lines of \cite{ST2019b} we prove the following (the detailed result is outlined in \thm{t:attr} below).

\begin{theorem}[{\bf Local flocking with attraction potential}]\label{t:iattr}
	Consider the multi-flock dynamics \eqref{e:iCSattr} with fat-tailed radial kernels, $\phi_\a(r)\gtrsim \langle r \rangle^{-\eta}$ and convex potential $U(r) \gtrsim r^\beta$ with tamed growth $U^{(k)}(r) \lesssim r^{\b-k}, \ k=1,2$, for some $\beta\geq 1$ (further outlined in \eqref{e:Uattr} below). There exists $\eta_\b$ specified in \eqref{e:gb}, such that for $\eta\leq \eta_\b$,  the dynamics of each flock admits  asymptotic aggregation, 
	$	\limsup_{t \to \infty} \cD_\a(t) \leq L$,
	and alignment decay
	\[
	\frac{1}{2N_\a}\sum_{i=1}^{N_\a} |\bv_{\a i}-\bV_\a|^2 \lesssim \frac{C_\d}{\lan t \ran^{1-\d}}, \quad \forall \d > 0, \qquad \a=1,2, \ldots, A.
	\]
\end{theorem}
It should be emphasized that the confining action of the attraction potential is assumed  to act only on far-field,  $r>L$, but otherwise is allowed to be `turned-off' for $U(r)=0, \ r\leq L$ as depicted in figure \ref{f:AA}.This offers an extension of the recent result \cite{ST2019a} for the case $L=0$. In fact, as noted in \thm{t:aggr} below, if the potential $U$ has a global support, then there is \emph{exponential rate} alignment.

\begin{figure}
			\includegraphics[width=4in]{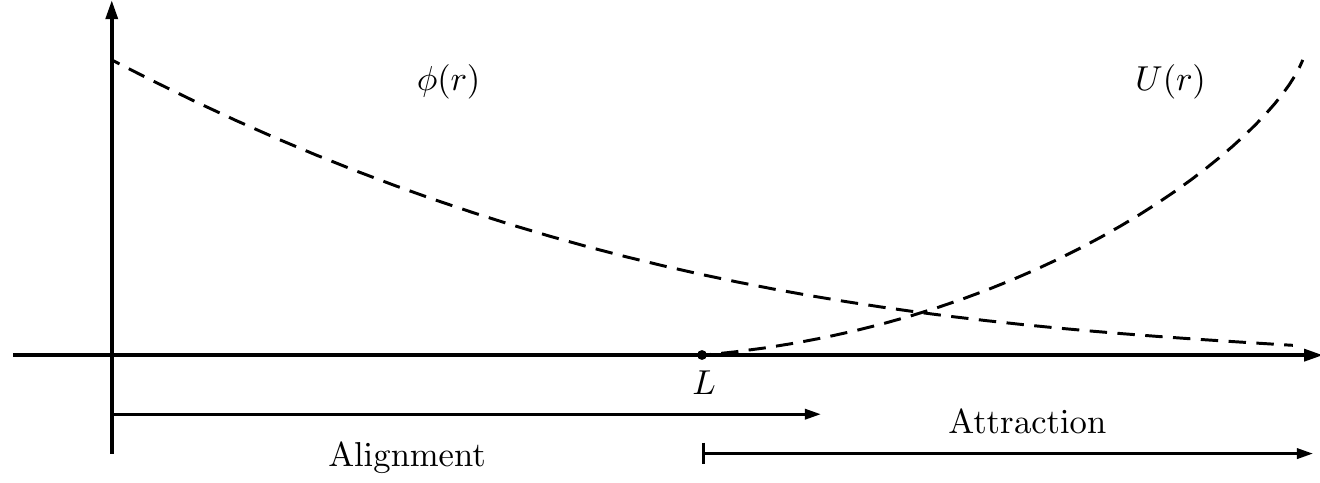}
		\caption{2-zone Attraction-Alignment model}\label{f:AA}
\end{figure}

When $N_\a \gg1$ one recovers the large-crowd dynamics in terms of the macroscopic density and velocity $(\rho_a,\bu_\a)$, governed by the hydrodynamic multi-flock system, which is the topic of section \ref{sec:hydro}
\[
\left\{
\begin{split}
\p_t \rho_\a + \n \cdot (\bu_\a \rho_\a) & = 0\\
\p_t {\bu}_{\a} + \bu_\a \cdot \n \bu_\a & =\l_\a \int_{\R^d} \phi_\a(\bx,\by)(\bu_\a(\by) - \bu_\a(\bx)) \rho_\a(\by) \dby  \\
&\ \ \ + \e \sum_{\b \neq \a} M_\b \psi(\bX_\a ,\bX_\b)(\bV_\b - \bu_\a(\bx,t)),
\end{split}\right. \qquad  \a = 1,\ldots,A.
\]
Here $\{(\bX_\a,\bV_\a)\}_\a$ are the macroscopic quantities which record the center of mass and momentum of $\a$-flock governed by \eqref{e:CSmacro}. The alignment dynamics reflects the discrete framework of Theorems \ref{t:fast} and \ref{t:slow}, namely --- if $\phi_\a$ and $\psi$ are fat-tailed then
 smooth solutions of the $\a$-flock and, respectively, the whole crowd will align towards their respective averages. The details can be found in \thm{t:hydroalign} below. In particular,  we prove that  the 
1D multi-flock hydrodynamics with  radial $\phi_\a$'s ---  either smooth or singular, and subject to sub-critical initial condition
$u'_\a(x,0) + \l_\a \phi_\a \ast \rho_\a(x,0)\geq 0$, $\forall x \in \R$,  admits  global smooth solution and flocking insues.\newline

%%%%%%%%%%%%%%%%%%%%%%%
\section{From agents to multi-flocks and back: up-scaling}\label{sec:multi}
%%%%%%%%%%%%%%%%%%%%%%
\subsection{Agent-based description} Our starting point is the alignment-based dynamics \eqref{eq:ial}
\begin{equation}\label{eq:al}
\begin{split}
\dot{\bx}_i(t)&=\bv_i(t)\\
\dot{\bv}_i(t) &= \sum_{j\in {\calC}} \phi(\bx_i,\bx_j)(\bv_j(t)-\bv_i(t)), \qquad i\in {\calC}:=\{1,2,\ldots, N\}.
\end{split}
\end{equation}
This expresses the tendency of agents  to align their velocities with the rest of the crowd, dictated by the \emph{symmetric} communication kernel $ \phi(\cdot,\cdot)\geq 0$. Let us assume that each of the terms on the right is  of the same   order, ${\mathcal O}(1)$;  then the total action on the right of order ${\mathcal O}(N)$ will peak at time $t={\mathcal O}(1/N)$. Using the   scaling parameter $\lambda=1/N$, one arrives at the celebrated Cucker-Smale model \cite{CS2007a,CS2007b}
\[
\dot{\bv}_i = \lambda\sum_{j\in {\calC}} \phi(\bx_i,\bx_j)(\bv_j-\bv_i), \qquad \lambda=\frac{1}{N},
\]  
where  the dynamics is re-scaled to peak at the desired  $t \sim {\mathcal O}(1)$. But what happens when the terms on the right of \eqref{eq:al} are of different order?  Assume that the crowd consists of two mostly separated groups, 
${\calC}={\calC}_1\cup {\calC}_2$, where ${\calC}_1$ has a large crowd of $N_1$ agents whereas ${\calC}_2$ has a much smaller crowd of $N_2 \ll N_1$ agents. By `mostly separated' we mean that the two groups have a very low level of communication so that $\{\phi(\bx_i,\bx_j) \ll 1 \ | \ (\bx_i,\bx_j)\in ({\calC}_1,{\calC}_2)\}$. We will quantify   a precise statement of separation in section \ref{sec:space} below.
Now the dynamics \eqref{eq:al} will experience two-time scales: the action of the larger crowd  ${\calC}_1$ will peak  earlier at time $t_1={\mathcal O}(1/N_1)$, mostly ignoring the negligible effect of the `far way' crowd in ${\calC}_2$. The crowd of ${\calC}_2$ will peak \emph{much later} at time $t_2={\mathcal O}(1/N_2) \gg t_1$. In \cite{MT2011} we suggested  an adaptive scaling parameter
\[
\dot{\bv}_i = \lambda_i\sum_{j\in {\calC}} \phi(\bx_i,\bx_j)(\bv_j-\bv_i), \qquad \lambda_i=\frac{1}{\sum_j \phi(\bx_i,\bx_j)},
\]  
here, $\lambda_i$  adapts itself to the different clocks of both crowds: when in the larger crowd $i\in {\calC}_1$, we have $\lambda_i \sim 1/N_1$ whereas for agents in the smaller crowd $i\in {\calC}_2$ we have $\lambda_i \sim 1/ N_2$
\[
\dot{\bv}_i = \left\{\begin{array}{ll}
\lambda_i\sum_{j\in {\calC}_1} \phi(\bx_i,\bx_j)(\bv_j-\bv_i),  & \displaystyle i\in {\calC}_1: \  \lambda_i=\frac{1}{\sum_j \phi(\bx_i,\bx_j)} \sim \frac{1}{N_1}\\
\lambda_i\sum_{j\in {\calC}_2} \phi(\bx_i,\bx_j)(\bv_j-\bv_i),  & \displaystyle i\in {\calC}_2: \  \lambda_i=\frac{1}{\sum_j \phi(\bx_i,\bx_j)} \sim \frac{1}{N_2}. \end{array}\right.
\] 
Thus,  $\lambda_i$ should be viewed as \emph{time scaling} adapted for  both crowds to peak at the desired $t={\mathcal O}(1)$. While this scaling is satisfactory for ${\calC}_1$, it neglects taking into account that the activity of the smaller  ${\calC}_2$ peaks much  later after the peak of the larger crowd ${\calC}_1$, which has an additional effect on the dynamics of ${\calC}_2$. 
\subsection{Scale separation in time} We want to take both groups into account while being precise of using the same `clock'. To this end, it will be convenient to observe the  configurations of crowds ${\calC}_1$  and ${\calC}_2$  in terms of their empirical distribution 
\[
\mu_{{}_{1}}(\bx,\bv,t):=\frac{1}{N_1}\sum_{k\in {\calC}_1} \delta_{\bx_k(t)}(\bx)\otimes \delta_{\bv_k(t)}(\bv), \quad 
\mu_{{}_{2}}(\bx,\bv,t):=\frac{1}{N_2}\sum_{k\in {\calC}_2} \delta_{\bx_k(t)}(\bx)\otimes \delta_{\bv_k(t)}(\bv).
\]
We distinguish between three time scales.

\medskip\noindent
(i) Time $t\lesssim t_1$. The dynamics is captured by the agent-based description
of the two separate groups which form the crowd ${\calC}$ in \eqref{eq:al}.

\medskip\noindent
(ii) Time $t_1 \ll t\lesssim t_2$. Since $t_2 \gg t_1$, crowd   ${\calC}_1$ is captured by its large-time dynamics which is realized  as a continuum with macroscopic density  $\mu_{{}_{1}}(\bx,\bv,t)\dbv \stackrel{N_1\gg1 }{\longrightarrow} \rho_1(\bx,t): \R^d\times \R_+ \rightarrow \R_+$, and  momentum $\mu_{{}_{1}}(\bx,\bv,t)\bv\dbv \stackrel{N_1\gg1}{\longrightarrow} (\rho_1\bu_1)(\bx,t):\R^d\times \R_+ \mapsto \R^d$. 
Observe that the dynamics at this stage  involves two groups  with two different descriptions: crowd ${\calC}_1$ is encoded in terms of its hydrodynamic observables, $(\rho_1,\rho_1\bu_1)$, while  crowd ${\calC}_2$ is still encoded in terms of its agent-based description
\[
\rho(\by,t) = \rho_1(\by,t) + \overbrace{\frac{1}{N_2} \sum_{k\in {\calC}_2} \delta_{\bx_k(t)}(\by)}^{\rho_2(\by,t)}, \qquad \rho\bu(\by,t) = \rho_1\bu_1(\by,t) + \overbrace{\frac{1}{N_2} \sum_{k\in {\calC}_2}\bv_k(t)\delta_{\bx_k(t)}(\by)}^{\rho_2\bu_2(\by,t)}.
\]
The large-time dynamics of ${\calC}_1$ is governed by the hydrodynamic system \cite{HT2008, CFTV2010}
\refstepcounter{equation}
\begin{equation}\label{e:main}\tag*{(\theequation)$_{1}$}
\ \ \ \ \ \ \left\{
\begin{split}
(\rho_1)_t + \nabla_\bx \cdot (\rho_1 \bu_1) & = 0, \\
(\rho_1\bu_1)_t + \nabla_\bx\cdot (\rho_1\bu_1 \otimes \bu_1 + P_1) &= \int_{\R^n}\!\!\!\phi(\bx,\by)\big\{(\rho\bu)(\by,t)\rho_1(\bx,t) - \rho(\by,t)(\rho_1\bu_1)(\bx,t))\big\}\dby,
\end{split}\right. 
\end{equation}
while crowd  ${\mathcal C}_2$ is governed by the agent-based description \eqref{eq:al} which takes the weak formulation
\begin{equation}\label{e:nain}\tag*{(\theequation)$_{2}$}
\ \ \ \ \ \ \left\{
 \begin{split}
(\rho_2)_t + \nabla_\bx \cdot (\rho_2 \bu_2) & = 0, \\
(\rho_2\bu_2)_t + \nabla_\bx\cdot (\rho_2\bu_2 \otimes \bu_2 + P_2) &= \int_{\R^n}\!\!\!\phi(\bx,\by)\big\{(\rho\bu)(\by,t)\rho_2(\bx,t) - \rho(\by,t)(\rho_2\bu_2)(\bx,t))\big\}\dby.
\end{split}\right. 
\end{equation}
Here, $P_1=P(\bv-\bu_1\otimes \bv-\bu_1)$ is a second-order fluctuations pressure tensor which requires a closure relations between the microscopic and macroscopic variables. We shall not dwell on its specific form:  the large time behavior of ${\mathcal C}_1$ in \ref{e:main}  is \emph{independent} of the specifics  of this closure. It will suffice to observe the center of mass and average velocity of crowd ${\mathcal C}_1$:
\[
\bX_1(t):= \frac{1}{M_1}\int_{{\mathcal S}_1}\bx\rho_1(\bx,t)\dbx,
\qquad \bV_1(t):= \frac{1}{M_1}\int_{{\mathcal S}_1}\rho_1(\bx,t)\bu_1(\bx,t)\dbx, \quad {\mathcal S}_1:=\text{supp}\{\rho_1(t,\cdot)\}.
\]
Integrating \ref{e:main} over the support of the first crowd ${\mathcal S}_1$: since the `self-'-alignment of ${\mathcal C}_1$ with itself  vanishes for $\by\in {\mathcal S}_1$, we are left with the contribution from the second crowd $\rho(\by,t) \mapsto \rho_2=\frac{1}{N_2} \sum_{k\in {\mathcal C}_2} \delta_{\bx_k(t)}(\by)$ and $(\rho\bu)(\by,t) \mapsto \rho_2\bu_2= \frac{1}{N_2} \sum_{k\in {\mathcal C}_2}\bv_k(t)\delta_{\bx_k(t)}(\by)$, which yields  
\[%begin{equation}\label{eq:C1}
\begin{split}
\dot{\bX}_1&=\bV_1\\
M_1\dot{\bV}_1& = \int_{\bx\in {\mathcal S}_1}\int_{\by\in {\mathcal S}_2}\phi(\bx,\by)\big\{(\rho_2\bu_2)(\by,t)\rho_1(\bx,t) - \rho_2(\by,t)(\rho_1\bu_1)(\bx,t))\big\}\dby\dbx \\
& =  \frac{1}{N_2} \sum_{j\in{\mathcal C}_2} \bv_{2j}(t) \int_{\bx\in {\mathcal S}_1}  \phi(\bx,\bx_{2j})\rho_1(\bx,t)\dbx - \frac{1}{N_2}\sum_{j\in{\mathcal C}_2}\int_{\bx\in {\mathcal S}_1}\phi(\bx,\bx_{2j})  (\rho_1\bu_1)(\bx,t)\dbx
\end{split}
\]%end{equation}
Due to assumed relatively large separation between the flocks, we can  approximate the last two integrals by the values of the kernel integrands at the centers of mass:
\begin{equation}\label{eq:approx1}
	\left\{
	\begin{split}
		\int_{\bx\in {\mathcal S}_1}  \phi(\bx,\bx_{2j})\rho_1(\bx,t)\dbx & =: \mu\psi(\bX_1,\bX_2)M_1, \\
		 \int_{\bx\in {\mathcal S}_1}\phi(\bx,\bx_{2j})  (\rho_1\bu_1)(\bx,t)\dbx & =: \mu\psi(\bX_1,\bX_2)M_1\bV_1,
	\end{split} \right. \qquad \mu\psi(\bX,\bY)\approx \phi(\bX,\bY), \ \ \mu\ll 1
\end{equation}
obtaining
\begin{subequations}\label{eqs:C12t2}
	\begin{equation}\label{eq:C1t2}
	\left\{
	\begin{split}
	\dot{\bX}_1(t)&=\bV_1(t)\\
	\dot{\bV}_1(t)& = \mu\psi(\bX_1,\bX_2)\big( \bV_2(t)-\bV_1(t)\big), \quad \bV_2(t)=\frac{1}{M_2}\int \rho_2\bu_2(\bx,t)\dbx= \frac{1}{N_2}\sum_{j\in {\mathcal S}_2}\bv_{2j}(t).
	\end{split}\right.
	\end{equation}
	For the dynamics of the second group ${\mathcal C}_2$ we may take $P_2\equiv 0$ on the left of \ref{e:nain}.  The cross-group interactions term $(\rho,\rho\bu)\mapsto (\rho_1,\rho_1\bu_1)$ on the right of \ref{e:nain} yields
	\[
	\int_{\by\in {\mathcal S}_1}  \phi(\bx_{2j},\by)\rho_1(\by,t)\dby = \mu\psi(\bX_2,\bX_1)M_1, \quad \int_{\by\in {\mathcal S}_1}\phi(\bx_{2j},\by)  (\rho_1\bu_1)(\by,t)\dby = \mu\psi(\bX_2,\bX_1)M_1\bV_1,
	\]
	arriving at 
	\begin{equation}\label{eq:C2t2}
	\left\{\quad
	\begin{split}
	\dot{\bx}_{2i}&=\bv_{2i}\\
	\dot{\bv}_{2i} &= \sum_{j\in {\mathcal C}_2} \phi(\bx_{2i},\bx_{2j})(\bv_{2j}-\bv_{2i})
	+ \mu\psi(\bX_2,\bX_1)M_1(\bV_1-\bv_{2i}).
	\end{split}\right.
	\end{equation}
\end{subequations}

Thus,  we end up with a new agent-based dynamics, \eqref{eq:C2t2}, in which the dynamics of group ${\mathcal C}_1$  is encoded as  new agent governed by mean position $\bX_1$ and  a mean   velocity $\bV_1$. This is a precisely the system \eqref{e:CS} written for the smaller flock ${\mathcal C}_2$.

\medskip\noindent
(iii) Time $t\gg t_2$. Now the second crowd  ${\mathcal C}_2$ is also captured by its large-time dynamics,   realized  in terms of macroscopic density  $\mu_{2}(\bx,\bv,t)\dbv \rightarrow \rho_2(\bx,t): \R^d\times \R_+ \mapsto \R_+$, and  momentum $\mu_{2}(\bx,\bv,t)\bv\dbv \rightarrow (\rho_2\bu_2)(\bx,t):\R^d\times \R_+ \mapsto \R^d$. Together, groups ${\mathcal C}_1$ and ${\mathcal C}_2$ form the crowd
\[
\rho(\by,t) = \rho_1(\by,t) + \rho_2(\by,t) \qquad \rho\bu(\by,t) = \rho_1\bu_1(\by,t) +\rho_2\bu_2(\by,t),
\]
which is governed by \ref{e:main}-\ref{e:nain}.
Here, $P_2=P(\bv-\bu_2\otimes \bv-\bu_2)$ is a second-order fluctuations pressure tensor which requires a closure relations between the microscopic and macroscopic variables. But we do  not dwell on its specific form, since   the large time behavior of ${\mathcal C}_2$ in \ref{e:nain}  is captured by the center of mass and average velocity of crowd ${\mathcal C}_2$:
\[
\bX_2(t):= \frac{1}{M_2}\int_{{\mathcal S}_2}\bx\rho_2(\bx,t)\dbx,
\qquad \bV_2(t):= \frac{1}{M_2}\int_{{\mathcal S}_2}\bu_2(\bx,t)\rho_2(\bx,t)\dbx, \quad {\mathcal S}_2:=\text{supp}\{\rho_2(\cdot,t)\}.
\]
Integrating \ref{e:nain} over the support of the first crowd ${\mathcal S}_2$: since the `self-'-alignment of ${\mathcal C}_2$ with itself  vanishes for $\by\in {\mathcal S}_2$, and using \eqref{eq:approx1} we are left with 
\[%begin{equation}\label{eq:C1}
\begin{split}
\dot{\bX}_2&=\bV_2,\\
M_2\dot{\bV}_2& = \int_{\bx\in {\mathcal S}_2}\int_{\by\in {\mathcal S}_1}\phi(\bx,\by)\big\{(\rho\bu)(\by,t)\rho_2(\bx,t) - \rho(\by,t)(\rho_2\bu_2)(\bx,t))\big\}\dby\dbx \\
& = \int_{\bx\in {\mathcal S}_2}\int_{\by\in {\mathcal S}_1}\phi(\bx,\by)\big\{(\rho_1\bu_1)(\by,t)\rho_2(\bx,t) - \rho_1(\by,t)(\rho_2\bu_2)(\bx,t))\big\}\dby\dbx
\\
& = \int_{\bx\in {\mathcal S}_2}\phi(\bx,\bX_1)\big\{M_1\bV_1\rho_2(\bx,t) - M_1(\rho_2\bu_2)(\bx,t))\big\}\dbx.
\end{split}
\]%end{equation}
We approximate the last two integrals by the same principle as before
\begin{equation}\label{eq:approx2}
	\begin{split}
		\int_{\bx\in {\mathcal S}_2}  \phi(\bx,\bX_1)\rho_2(\bx,t)\dbx & = \mu\psi(\bX_2,\bX_1)M_2,\\ 
		 \int_{\bx\in {\mathcal S}_2}\phi(\bx,\bX_1)  (\rho_2\bu_2)(\bx,t)\dbx &= \mu\psi(\bX_2,\bX_1)M_2\bV_2,
	\end{split}
\end{equation}
arriving at a  2-agent system described by the dynamics of their center of mass/momentum $(\bx_\a,\bV_\a)$, 
\begin{equation}\label{eq:C12ggt2}
\left\{\begin{split}
\quad \dot{\bX}_\a(t)&=\bV_\a(t),\\
\quad M_\a\dot{\bV}_\a(t)& = \mu\sum_{\b \neq \a} \psi(\bX_\a,\bX_\b)M_\a M_\b (\bV_\b(t)-\bV_\a(t)),
\end{split}\right. \qquad \a,\b \in\{1,2\}.
\end{equation}

 In summary, we began with the  agent based description for two crowds of $N_1\gg N_2$ agents, \eqref{eq:al} valid for $t\lesssim t_1$. It evolved into an agent-based  description  for  crowd of $N_2+1$ agents \eqref{eqs:C12t2} valid for $t_1 \ll t \lesssim  t_2$ and ended with 2-agent description \eqref{eq:C12ggt2} valid for $t\gg t_2$.
This is a process of \emph{up-scaling} in which the notion of an `agent' is replaced with a `multi-flock' --- a larger blob made of agents, which is identified by its center of mass/momentum. The only difference is that the multi-flock-based dynamics now takes into account only the up-scaled quantities of the multi-flock. Let us recall that the more general system \eqref{e:CS} permits up-scaling in the same way. 

\subsection{Scale separation in space}\label{sec:space} Following up on the idea of spatial separation between islands it is instructive to  assess the scale on which  approximation of mass/momentum given in \eqref{eq:approx1},\eqref{eq:approx2} is valid. To make analysis more precise we assume the large distance behavior of the communication kernel $\phi(\bx,\by) \sim |\bx-\by|^{-\eta}$. We consider the prototypical integrals in \eqref{eq:approx1}
$\displaystyle \int_{\bx \in {\mathcal S}_1}\phi(\bx,\by)\rho_1(\bx,t)\dbx$ and
$\displaystyle \int_{\bx \in {\mathcal S}_1}\phi(\bx,\by)\rho_1\bu_1(\bx,t)\dbx$
for $\by\in {\mathcal S}_2$.
We now fix $\bX\in \conv{\mathcal S}_1$ and $\bY \in \conv{\mathcal S}_2$, and for any given pair of agents $\bx\in{\mathcal S}_1$, $\by\in{\mathcal S}_2$ we decompose  $\bx-\by=(\bX -\bY)+ (\bY - \by) - (\bX-\bx)$. Thus, $R:=|\bX-\bY|$ is the (fixed)  long-range distance between  the two groups, whereas
$r:= |(\bY - \by) - (\bX-\bx)|$ encapsulates the short-range distances within the crowds, $r\ll R$. Similar decomposition holds for the weighted integral of $\rho_2$ sought in \eqref{eq:approx2}. We have
\[
\frac{1}{|\bx-\by|} = \frac{1}{\sqrt{R^2+r^2-2r\cos\theta}}=
\frac{1}{R}\sum_{k=0}^\infty\left(\frac{r}{R}\right)^kP_k(\cos \theta), 
\]
where $\cos\theta= \big\langle (\bX-\by)/R, (\bX-\bx)/r\big\rangle$ and  $P_k$ are the $k$-degree Legendre polynomials, $P_0(x)=1, P_1(x)=x$ etc.
We find
\[
\begin{split}
\frac{1}{|\bx-\by|}&= \frac{1}{R} + \frac{r}{R^2}\Big\langle \frac{\bX-\bY}{R}, \frac{(\bY - \by) - (\bX-\bx)}{r}\Big\rangle +{\mathcal O}\left(\frac{r^2}{R^3}\right)\\
& = \frac{1}{R}\left(1+ \frac{1}{R^2}\Big\langle \bX-\bY, (\bY - \by) - (\bX-\bx)\Big\rangle
+ {\mathcal O}\left(\frac{r^2}{R^2}\right)\right), \qquad \bx \in{\mathcal S}_1, \ \ \by\in{\mathcal S}_2.
\end{split}
\]
Since the contribution of the second term on the right  is of order $r/R \ll 1$ we can further approximate
\[
\begin{split}
\phi(\bx,\by) &\sim \frac{1}{|\bx-\by|^\eta} =
\frac{1}{R^\eta} \left(1+\frac{\eta}{R^2}\Big\langle \bX-\bY, (\bY - \by) - (\bX-\bx)\Big\rangle +{\mathcal O}\left(\frac{r^2}{R^2}\right)\right) \\
& = \frac{1}{R^\eta}+ \frac{\eta}{R^{2+\eta}}\big\langle \bX-\bY, (\bY - \by) - (\bX-\bx)\big\rangle +{\mathcal O}\left(\frac{r^2}{R^{2+\eta}}\right).
\end{split}
\]
The first key point is that by choosing $\bX=\bX_1$ and $\bY=\bX_2$ as the centers of mass of the flocks, so that 
$\displaystyle M_1\bX_1=\int_{\bx \in {\mathcal S}_1}\bx\rho_1(\bx,t)$,
then the second term  has a negligible contribution. Indeed, 
\[
\begin{split}
\int_{\bx \in {\mathcal S}_1}\phi(\bx,\by)&\rho_1(\bx,t)\dbx
\sim  \int_{\bx \in {\mathcal S}_1} \frac{1}{|\bx-\by|^\eta}\rho_1(\bx,t)\dbx
\\
= & \frac{1}{R^\eta} \int_{\bx \in {\mathcal S}_1}\rho_1(\bx,t)\dbx + \frac{\eta}{R^{2+\eta}}\int_{\bx \in {\mathcal S}_1} \langle \bX_1-\bX_2, (\bX_2-\by) - (\bX_1 - \bx)\rangle \rho_1(\bx,t)\dbx +{\mathcal O}\left(\frac{r^2}{R^{2+\eta}}\right)\\
& = \frac{1}{R^\eta}M_1 +  \frac{\eta}{R^{2+\eta}}\int_{\bx \in {\mathcal S}_1} \langle \bX_1-\bX_2, (\bX_2-\by) \rangle \rho_1(\bx,t)\dbx  +{\mathcal O}\left(\frac{r^2}{R^{2+\eta}}\right)\\
& = \frac{1}{R^\eta}M_1 + \cO\left( \frac{r}{R^{1+\eta}}\right)+   {\mathcal O}\left(\frac{r^2}{R^{2+\eta}}\right).
\end{split}
\]
Noting that $\phi(\bX_1,\bX_2)=  R^{-\eta}$ we conclude with the first part of \eqref{eq:approx1}
\begin{subequations}\label{eqs:approx1rhoandm}
	\begin{equation}\label{eq:approx1rho}
	\int_{\bx \in {\mathcal S}_1}\phi(\bx,\by)\rho_1(\bx,t)\dbx = \phi(\bX_1,\by)M_1 + {\mathcal O}\left(\frac{r}{R^{1+\eta}}\right), \qquad \by\in{\mathcal S}_2.
	\end{equation}
Similarly, we recover the asymptotic formula for momentum \eqref{eq:approx1}
	\begin{equation}\label{eq:approx1m}
	\int_{\bx \in {\mathcal S}_1}\phi(\bx,\by)(\rho_1\bu_1)(\bx,t)\dbx = \phi(\bX_1,\bX_2)M_1\bV_1(t) + {\mathcal O}\left(\frac{r}{R^{1+\eta}}\right).
	\end{equation}
\end{subequations}
The same argument applies for crowd ${\mathcal C}_2$: 
\begin{equation}\label{eq:approx2rhoandm}
\begin{split}
\int_{\bx \in {\mathcal S}_2}\phi(\bx,\by)\left\{\begin{array}{l}\rho_2(\bx,t) \\
(\rho_2\bu_2)(\bx,t)\end{array}\right\} \dbx &= \phi(\bX_2,\bX_1)\left\{\begin{array}{l}M_2\\ M_2\bV_2(t)\end{array}\right\} + {\mathcal O}\left(\frac{r}{R^{1+\eta}}\right), \quad \by\in{\mathcal S}_1.
\end{split}
\end{equation}
\begin{remark} The bounds  \eqref{eqs:approx1rhoandm},\eqref{eq:approx2rhoandm} quantify first order errors,  ${\mathcal O}(\epsilon_{ij}) \ll 1$, provided the diameters of  crowds ${\mathcal C}_i, {\mathcal C}_j$ are much smaller than their distance, $\epsilon_{ij}:=\max\{r_i,r_j\}/R_{ij} \ll 1$.
\end{remark}

\section{Slow and fast alignment in multi-flocks}\label{sec:rate}

\ifx%%%%%%%%%%%%%%%%%%%%%%%%%%%
Before we proceed to discussion of alignment and flocking  let us recall the Lyapunov function approach  of Ha and Liu \cite{HL2009}.

\subsection{A system of ordinary differential inequalities and its Lyapunov function}

 Let us consider the following  system of ordinary differential inequalities for an abstract pair of parameters $(V,D)$:
\begin{equation}\label{e:ODI}
\left\{
\begin{split}
\dot{A} & \leq - \l \psi(D)A \\
\dot{D} & \leq  A.
\end{split}\right.
\end{equation}

\begin{lemma}[\cite{HL2009}]\label{l:HL} Suppose $\psi$ is a smooth, non-increasing, everywhere positive kernel. Consider a solution to \eqref{e:ODI} with initial condition $(A_0,D_0)$. If the kernel satisfies the following condition 
\begin{equation}\label{e:fat-sid}
\int_{D_0}^{\infty} \psi (r) \dr > \frac{A_0}{\l} ,
\end{equation} 
then  there exist $\overline{D}$ and $\d>0$ such that 
$D(t) \leq \overline{D}$ and $A(t) \leq A_0 e^{-\d t}$.
\end{lemma}
Note that condition \eqref{e:fat-sid} is always satisfied for the so-called fat tail kernels:
\begin{equation}\label{e:fatPsi}
\int_0^\infty \psi(r) \dr = \infty.
\end{equation} 
\begin{proof} The result follows easily by noticing that the  system \eqref{e:ODI}  has a decreasing Lyapunov function given by $	L (A,D) = A + \l  \int_0^{D} \psi (r) \dr$. This, in particular, implies that 
	\[
\l \int_0^{D(t)} \psi (r) \dr \leq  A_0 + \l\int_0^{D_0} \psi (r) \dr, \quad \forall t>0.
	\]
Consequently, $D(t) \leq \bar{D}$, where $\bar{D}$ is obtained from the equation
	\begin{equation}\label{e:Dinfty}
\l \int_{D_0}^{\bar{D}} \psi (r) \dr =  A_0,
	\end{equation}
which is guaranteed to have a finite solution due to   \eqref{e:fat-sid}.  Then,  $\dot{A} \leq - \l  \psi(\bar{D}) A$, hence
	\begin{equation}\label{e:Valign}
	A(t) \leq A_0 e^{-\l \psi(\bar{D}) t}.
	\end{equation}
\end{proof}

Solving equation \eqref{e:Dinfty} allows in some case to provide explicit decay rates for solutions of \eqref{e:ODI}, and obtain estimates on the diameter $\bar{D}$. In particular, for the classical Cucker-Smale kernel
\[
\psi(r) = \frac{1}{(1+r^2)^\frac{\g}{2}},
\]
one obtains
\begin{equation}\label{e:rate<1}
	\begin{split}
	\bar{D} & \leq \left( \left[  \frac{1-\g}{\l} V_0+ (1+D_0^2)^\frac{1-\g}{2}   \right]^{\frac{2}{1-\g}} - 1 \right)^\frac12, \quad \g<1, \\
	\mbox{rate} & = \frac{\l}{ \left[  \frac{1-\g}{\l} V_0 + (1+D_0^2)^\frac{1-\g}{2}   \right]^{\frac{\g}{1-\g}} }.
	\end{split}
\end{equation}
(here one replaces $\psi$ with a smaller but explicitly integrable kernel $ \frac{ r}{(1+r^2)^\frac{\g+1}{2}}$),and
\begin{equation}\label{e:rate1}
	\bar{D}  \leq  \left( e^{ \frac{2}{\l} V_0 } (1+D_0^2) - 1 \right)^\frac12, \quad \mbox{rate}  = \frac{\l}{  e^{ \frac{ V_0}{\l} } (1+D_0^2)^\frac12}, \quad \g = 1.
\end{equation}
These formulas will be helpful to understand dependency of the rate on communication strength within and between flocks.

Let us consider the following size metrics of the system \eqref{e:CS}:
\[
\begin{split}
 \cD_\a &= \max_{i,j}|\bx_{\a i} - \bx_{\a j}|, \quad  \cD = \max_{\a,\b}|\bX_{\a} - \bX_{\b}| \\
 \cA_\a &=  \max_{i,j } |\bv_{\a i}-\bv_{\a j}|, \quad \cA =  \max_{\a,\b}   |\bV_{\a}-\bV_{\b}|.
\end{split}
\]
\fi%%%%%%%%%%%%%%%%%%%%%%%%%

In this section we focus on alignment dynamics for system \eqref{e:CS} under conditions of \thm{t:fast} and \ref{t:slow}. In fact, with a slight abuse of notation we will make a more general assumption that there exist  radially symmetric subkernels
\begin{equation}\label{e:subk}
	\phi_\a(\bx,\by) \geq \phi_\a(\radial{\bx-\by}), \qquad \psi(\bx,\by)\geq \psi\radial{\bx-\by}),
\end{equation}
which are positive, monotonely decreasing, and fat tail at infinity
\begin{equation}\label{e:fatPsi}
\int^\infty_{r_0} \phi_\a(r)dr=\infty , \qquad \int^\infty_{r_0} \psi(r) dr=\infty.
 \end{equation}

We start by noting that any cluster system \eqref{e:CS} satisfies the global maximum principle -- maximum of each coordinate in the total family $\bv_{\a i}$ is non-increasing, and the minimum is non-decreasing. Therefore the system \eqref{e:CS} is well prepared ``as is" for establishing global flocking behavior.  However, this is not be the case for each individual flock. Each flock satisfies  "internal maximum principle" relative to its own time-dependent  momentum $\bV_\a$. This dictates passage to the reference frame evolving with that momentum and  center of mass:
\begin{equation}\label{e:wv}
\bw_{\a i } = \bv_{\a i} - \bV_\a, \quad \by_{\a i } =  \bx_{\a i} - \bX_\a.
\end{equation}
Using \eqref{e:CS} and \eqref{e:CSmacro} one readily obtains the system
\begin{equation}\label{e:CSshift}
\left\{
\begin{split}
\dot{\by}_{\a i}&=\bw_{\a i},\\
\dot{\bw}_{\a i}& = \l_\a  \sum_{j = 1}^{N_\a} m_{\a j}  \phi_{\a i j} (\bw_{\a i}-\bw_{\a j})- \e R_\a(t) \bw_{\a i},
\end{split}\right. 
\end{equation}
where $R_\a(t)  := \sum_{\b \neq \a} M_\b \psi(\radial{\bX_\a -\bX_\b})$, and we abbreviate $\phi_{\a i j}  = \phi_\a(\by_{\a i}+\bX_\a, \by_{\a j}+ \bX_\a)$. This system now does have a maximum principle and is well prepared for establishing flocking.

Let us denote individual flock parameters:
\[
\cD_\a(t) :=\max_{i, j=1,\ldots,N_\a}|\bx_{\a i}(t) - \bx_{\a j}(t)|, \quad \cA_\a = \max_{i, j=1,\ldots,N_\a}   | \bw_{\a i}-\bw_{\a j}| = \max_{\substack{\boldsymbol{\ell} \in \R^n: |\boldsymbol{\ell}| = 1\\ i, j=1,\ldots,N_\a}}  \lan \boldsymbol{\ell}, \bw_{\a i}-\bw_{\a j}\ran.
\]
By Rademacher's lemma,  we can evaluate the derivative of $\cA_\a$ by considering $\boldsymbol{\ell}, i, j$ at which that maximum is achieved at any instance of time:
\[
\begin{split}
\ddt \cA_\a & = \lan \boldsymbol{\ell},  \dot{\bw}_{\a i}- \dot{\bw} _{\a j}\ran  =  \l_\a  \sum_{k = 1}^{N_\a} m_{\a k}  \phi_{\a i k}  \lan \boldsymbol{\ell}, \bw_{\a k}-\bw_{\a i} \ran - \l_\a  \sum_{k = 1}^{N_\a} m_{\a k}  \phi_{\a j k}  \lan \boldsymbol{\ell}, \bw_{\a k}-\bw_{\a j} \ran \\
&- \e R_\a(t) \lan \boldsymbol{\ell}, \bw_{\a i}-\bw_{\a j}\ran \\
& =  \l_\a  \sum_{k = 1}^{N_\a} m_{\a k}  \phi_{\a i k} ( \lan \boldsymbol{\ell}, \bw_{\a k}-\bw_{\a j} \ran - \lan \boldsymbol{\ell}, \bw_{\a i}-\bw_{\a j} \ran ) \\
&+  \l_\a  \sum_{k = 1}^{N_\a} m_{\a k}  \phi_{\a j k} (  \lan \boldsymbol{\ell}, \bw_{\a i}-\bw_{\a k} \ran -  \lan \boldsymbol{\ell}, \bw_{\a i}-\bw_{\a j} \ran ) - \e R_\a(t)\cA_\a.
\end{split}
\]
Each difference of the action of $\boldsymbol{\ell}$ is negative due to maximality of $\boldsymbol{\ell}, i , j$. Hence, we replace values of $\phi_\a$'s with the use of \eqref{e:subk} and its minimal value at  $\cD_\a$:
\[
\begin{split}
\ddt \cA_\a &\leq  \l_\a  \phi_\a(\cD_\a) \sum_{k = 1}^{N_\a} m_{\a k}   ( \lan \boldsymbol{\ell}, \bw_{\a k}-\bw_{\a j} \ran - \lan \boldsymbol{\ell}, \bw_{\a i}-\bw_{\a j} \ran +  \lan \boldsymbol{\ell}, \bw_{\a i}-\bw_{\a k} \ran -  \lan \boldsymbol{\ell}, \bw_{\a i}-\bw_{\a j} \ran ) \\
& - \e R_\a(t)\cA_\a = - \l_\a  M_\a\phi_\a(\cD_\a) \cA_\a  - \e R_\a(t)\cA_\a.
\end{split}
\]
At the same time,  $R_\a(t) \geq M \psi(\cD)$, where
\[
\cD :=\max_{\a,\b}   |\bX_{\a}-\bX_{\b}|, \quad \cA :=\max_{\a,\b}   |\bV_{\a}-\bV_{\b}|.
\]
Combining it with the system for $(\cD,\cA)$ which follows a similar computation applied to macroscopic values \eqref{e:CSmacro},  we arrive at the following system of ODIs:
\begin{equation}\label{e:VDVDa}
\left\{\begin{aligned}
\dot\cA_\a & \leq - \l_\a M_\a \phi_\a(\cD_\a) \cA_\a  - \e M \psi(\cD)\cA_\a \\
\dot{\cD_\a} & \leq  \cA_\a\\
\dot{\cA} & \leq -  \e M \psi(\cD)\cA \\
\dot{\cD} & \leq  \cA.
\end{aligned}\right.
\end{equation}
This system encompasses prototypical systems of the form
\begin{equation}\label{e:ODI}
\left\{
\begin{split}
\dot{A} & \leq - \g \phi(D)A \\
\dot{D} & \leq  A,
\end{split}\right.
\end{equation}
Following Ha and Liu \cite{HL2009} we can define a Lyapunov function
\[
L = A + \g \int_0^D \phi(r)\dr,
\]
which is non-increasing.  Hence, there exists $\overline{D}$ and $\d>0$ such that
	\begin{equation}\label{e:fat-sid}
	\int_{D_0}^{\infty} \phi (r) \dr > \frac{A_0}{\g} \ \ \leadsto \ \ D(t) \leq \overline{D}, \ A(t) \leq A_0 e^{-\g \phi(\overline{D}) t}.
	\end{equation} 
Note that condition \eqref{e:fat-sid} is always satisfied for a  fat tail $\phi$.

Going back to \eqref{e:VDVDa} and ignoring the  term $ - \e M \psi(\cD)\cA_\a $ in the $\cA_\a$ equation we observe  that the $\a$-flock completely decouples from the rest of the multi-flock.   We arrive at \eqref{e:ODI} for the pair $(\cD_\a,\cA_\a)$. One obtains the fast internal alignment result \eqref{eq:fast} asserted in Theorem \ref{t:fast}
\[
\max_{i}| \bv_{\a i}(t) - \bV_\a(t)| \lesssim e^{-\d_\a t}, \qquad \d_\a=\l_\a M_\a \phi_\a(\overline{\cD}_\a).
\]
As noted before,  this indicates that the $\a$-flock behavior  depends solely only  on its own parameters, but not the global values.  In particular,  the $a$-flock alignment towards  $\bV_\a(t)$ occurs regardless whether these centers of momentum  align or not. The latter will be guaranteed if the inter-flock communication $\psi$ satisfies the fat tail condition \eqref{e:fatPsi}. In fact, in this case the global alignment ensues even if internal communications are  completely absent. This is evident from \eqref{e:VDVDa} where we ignore the $- \l_\a M_\a \phi_\a(\cD_\a) \cA_\a $ term and obtain boundedness of $\cD$ from the last two equations, obtaining the slow alignment \eqref{eq:slow} asserted in theorem \ref{t:slow} 
\[
\max_{\a, i}| \bv_{\a i}(t) - \bV| \lesssim e^{-\d t}, \qquad \d=\e M \psi(\overline{\cD}).
\]
Alignment rate in this case is slow since it depends on $\e$ and the global diameter of the multi-flock.

\begin{remark}[{\bf Asymptotic rate}]
Asymptotic dependence of the implied alignment rates for small $\e$  and large $\l_\a$ for the Cucker-Smale kernel can  be worked out from \eqref{e:fat-sid} (we omit the details). In the context of fast local alignment with $\phi_\a(r) \sim r^{-\eta_a}$ we obtain $\d \sim \l_\a$ for all $\eta_\a \leq 1$, while in the context of slow alignment with $\psi(r)=r^{-\zeta}$ we obtain  
\[
\d \sim \left\{\begin{array}{ll}\e^{\frac{1}{1-\zeta}},  & \zeta<1,\\
 \e e^{-1/\e}, \ & \zeta =1 .\end{array}\right.
 \]
\end{remark}

%%%%%%%%%%%%%%%%%%%%%%%%%%%%%%
\section{Multi-flocking driven by alignment and attraction}\label{sec:attr}
%%%%%%%%%%%%%%%%%%%%%%%%%%%%%
In this section we consider multi-flock alignment model with additional attraction forces. Out goal is to show that each flock would aggregate towards its center of mass within the radius of influence of the potential. Our results present an extension of \cite{ST2019b}.

We assume that the interactions are determined by a radially symmetric smooth potential $U \in C^2(\R_+)$:
\begin{equation}\label{e:CSattr}
\left\{
\begin{split}
\dot{\bx}_{\a i}&=\bv_{\a i},\\
\dot{\bv}_{\a i}& = \frac{1}{N_\a} \sum_{j = 1}^{N_\a}  \phi_\a(\radial{\bx_{\a i}- \bx_{\a j}}) (\bv_{\a j}-\bv_{\a i}) + \e \sum_{\substack{\b=1 \\ \b \neq \a}}^{A} \psi(\radial{\bX_\a-\bX_\b}) (\bV_{\b}-\bv_{\a i}) + \bF_\a(t),
\end{split}\right. 
\end{equation}
where
\[
\bF_{\a i}(t)= - \frac{1}{N_\a} \sum_{j=1}^{N_\a} \n U(\radial{\bx_{\a i} - \bx_{\a j}}).
\]
Here we assumed for notational simplicity that all masses are $1/N_\a$, and potentials are the same. However, the statements below can easily be carried out for a general set of parameters. 

Note that the system upscales to the same Cucker-Smale system \eqref{e:CSmacro} for the flock-level variables.

Using transformation \eqref{e:wv}, we rewrite the system in the new coordinate frame
\begin{equation}\label{e:CSshiftattr}
\left\{
\begin{split}
\dot{\by}_{\a i}&=\bw_{\a i},\\
\dot{\bw}_{\a i}& = \frac{1}{N_\a} \sum_{j = 1}^{N_\a}  \phi_{\a i j} (\bw_{\a i}-\bw_{\a j})- \e R_\a(t) \bw_{\a i} + \bF_{\a i}(t).
\end{split}\right. 
\end{equation}
The classical energy $\cE_\a  = \cK_\a + \cP_\a$ where\footnote{Here and in the sequel we occasionally use a shortcut for  $\ba_{\a ij} = \ba_{\a i} - \ba_{\a j}$.},
\begin{equation}\label{e:EKP}
\begin{split}
\cK_\a &:= \frac{1}{2N_\a} \sum_{i=1}^{N_\a}|\bw_{\a i}|^2 =  \frac{1}{4N_\a^2} \sum_{i=1}^{N_\a} |\bw_{\a ij}|^2,  \qquad \bw_{\a ij}=\bw_{\a i}-\bw_{\a j},\\
\cP_\a  &:= \frac{1}{2 N_\a^2} \sum_{i,j=1}^{N_\a} U(\radial{\by_{\a ij}}), \qquad
\by_{\a ij}=\by_{\a i}-\by_{\a j},
\end{split}
\end{equation}
satisfies
\begin{equation}\label{e:enlawU}
\ddt \cE_\a = -  \frac{1}{N_\a^2} \sum_{i,j =1}^{N_\a} \phi_{\a ij} |\bw_{\a ij}|^2  - \e R_\a(t) \cK_\a: = -\cI_\a -  \e R_\a(t) \cK_\a.
\end{equation}
At this stage already we can see that if $\e>0$ and $\psi$ has a fat tail, then global slow exponential alignment will insue regardless of internal flock communications. Indeed, the up-scaled dynamics \eqref{e:CSmacro} will stabilize the macroscopic values which implies boundedness of $R_\a$. Hence, ignoring dissipation term $\cI_\a$ in \eqref{e:enlawU} we obtain exponential decay of all the energies:
\[
\cE_\a \lesssim e^{- c \e t}.
\]
In this section we show that flocking occurs also in each individual $\a$-flock regardless of global communication, although it may be happening at a slower rate. To fix the notation we consider regular communication kernels with power-like decay:
\begin{equation}\label{e:kerUattr}
\phi_\a'(r) \leq 0, \quad  \phi_\a(r) \ge \frac{c_0}{\lan r \ran ^\g}, \text{ for } r \geq 0.
\end{equation}
For the potential we assume essentially a power law: for some $\b > 1$ and $L'>L > 0$, 
\begin{equation}\label{e:Uattr}
\begin{split}
\mbox{Support:}& \qquad U \in C^2(\R^+), \quad  U(r) = 0, \quad \forall r\leq L,\\
\mbox{Growth:}& \qquad U(r) \geq a_0 r^\b, \quad  |U'(r)| \leq a_1 r^{\b-1}, \quad |U''(r)| \leq a_2 r^{\b-2}, \quad \forall r>  L', \\
\mbox{Convexity:}& \qquad  U'(r),U''(r) \geq 0,\quad \forall r>0.
\end{split}
\end{equation}

\begin{theorem}[Local flocking with interaction potential]\label{t:attr}
	Under the assumptions \eqref{e:kerUattr} and \eqref{e:Uattr} on the kernel and potential in the range of parameters given by
	\begin{equation}\label{e:gb} 
	\g < \left\{
	\begin{split}
	1,& \quad 1<\b<\frac43,\\
	\frac32 \b -1,& \quad  \frac43\leq \b<2,\\
	2,&  \quad  \b\geq 2,\\
	\end{split}\right.
	\end{equation}
	all solutions to the system \eqref{e:CSattr} flock 	with the bound independent of $N_\a$:
		\[%\begin{equation}\label{key}
	\cD_\a(t) < \overline{\cD}_\a, \quad \forall t>0,
	\]%\end{equation}
 asymptotically aggregate 
	\[
	\limsup_{t \to \infty} \cD_\a(t) \leq L,
	\]
	and align
	\begin{equation}\label{e:driniE}
	\frac{1}{2N_\a} \sum_{i=1}^{N_\a}|\bv_{\a i}-\bV_\a|^2 + \frac{1}{2 N_\a^2} \sum_{i,j=1}^{N_\a} U(\radial{\bx_{\a i}-\bx_{\a j}}) \leq \frac{C_\d}{\lan t \ran^{1-\d}}, \quad \forall \d > 0.
	\end{equation}
\end{theorem}

Note that the latter statement follows from local alignments \eqref{e:driniE} and the exponential alignment of the flock parameters governed by the upscaled system \eqref{e:CSmacro}.

\begin{proof}
	We will operate with the particle energy defined similarly to \cite{ST2019b}
	\[%\begin{equation}\label{key}
	\cE_{\a i} = \frac12 |\bw_{\a i}|^2 + \frac{1}{N_\a} \sum_{k=1}^{N_\a} U(\radial{\by_{\a ik}}),\quad \cE_{\a \infty} = \max_i \cE_{\a i}.
	\]%\end{equation}
	First, we observe that the particle energy controls the diameter of the flock. Indeed, by convexity and our assumptions on the growth of the potential,  we have
	\begin{equation}\label{e:Ei0}
	\cE_{\a i} \geq U(\radial{\by_{\a i}}) \geq (|\by_{\a i}| - L')_+^\b.
	\end{equation}
	So, 
	\begin{equation}\label{e:diamE}
	\cD_\a \leq   \cE_{\a \infty}^{1/\b} + L'.
	\end{equation}
Let us now establish a bound on $ \cE_{\a \infty}$.  For each $i$ we test \eqref{e:CSshiftattr} with $\bw_{\a i}$ and ignore that $R_\a$-term :
	\begin{equation}\label{e:Ei1}
	\ddt \cE_{\a i } \leq  \frac{1}{N_\a} \sum_{k=1}^{N_\a} \phi_{\a ik} \bw_{\a ki} \cdot \bw_{\a i} - \frac{1}{N_\a} \sum_{k=1}^{N_\a} \n U(\radial{\by_{\a ik}}) \cdot \bw_{\a k}.
	\end{equation}
	For the kinetic part we use the vector identity
	\begin{equation}\label{e:Ei2}
	\ba_{ki} \cdot \ba_i= - \frac12 |\ba_{ki}|^2 - \frac12 |\ba_i|^2 + \frac12 |\ba_k|^2.
	\end{equation}
	Discarding all the negative terms, we bound 
	\[
\frac{1}{N_\a} \sum_{k=1}^{N_\a} \phi_{\a ik} \bw_{\a ki} \cdot \bw_{\a i}  \leq  |\phi_\a|_\infty \cK_\a.
	\]
	Due to the energy law $\cK_\a$ will remain bounded, but we will keep it in the bound above for now.  As for the potential term, there are several ways we can handle it. 
	
	For any $1\leq \b\leq \frac43$ we apply a direct estimate from the first derivative:
	\[
	\left| \frac{1}{N_\a} \sum_{k=1}^{N_\a}  \n U(\radial{\by_{\a ik}}) \cdot \bw_{\a k} \right| \leq \sqrt{\cK_\a}  \left(  \frac{1}{N_\a} \sum_{k=1}^{N_\a}  | \n U(\radial{\by_{\a ik}}) |^2 \right)^\frac12 \leq \sqrt{\cK_\a} \cD_\a^{\b-1}.
	\]
	Consequently,
	\[
	\ddt \cE_{\a i} \leq c_1 \cK_\a + c_2  \sqrt{\cK_\a} \cD_\a^{\b-1} \lesssim   \sqrt{\cK_\a}(1+ \cE_{\a \infty}^{\frac{\b-1}{\b}}),
	\]
	and 
	\begin{equation}\label{e:parE1}
	\ddt \cE_{\a \infty} \leq c_3 \sqrt{\cK_\a}(1+ \cE_\infty^{\frac{\b-1}{\b}})  \quad \Rightarrow \quad  \cE_{\a \infty}  \lesssim    \lan t\ran^{\b}  \quad \Rightarrow \quad \cD_\a \lesssim \lan t\ran.
	\end{equation}

	In the range $\frac43 \leq \b \leq 2$ it is better to make use of the second derivative:
	\begin{equation}\label{e:potS}
	\begin{split}
	\left| \frac{1}{N_\a} \sum_{k=1}^{N_\a} \n U(\radial{\by_{\a ik}}) \cdot \bw_{\a k} \right| & = \frac{1}{N_\a} \sum_{k=1}^{N_\a}  (\n U(\radial{\by_{\a ik}}) - \n U(\radial{\by_{\a i}}) ) \cdot \bv_k \\
	& \leq \|D^2 U\|_\infty \sqrt{\cK_\a} \left(  \frac{1}{N_\a} \sum_{k=1}^{N_\a}  | \by_{\a k} |^2 \right)^\frac12 \\
	 & \leq  c_4 \sqrt{\cK_\a} \left(  \frac{1}{N_\a^2} \sum_{i,j=1}^{N_\a} | \by_{\a ij} |^2 \right)^\frac12 .
	\end{split}
	\end{equation}
	The following inequality will be used repeatedly
	\begin{equation}\label{e:sqP}
 \frac{1}{N_\a^2} \sum_{i,j=1}^{N_\a} | \by_{\a ij} |^2 \leq (L')^2 + 	 \frac{1}{N_\a^2} \sum_{i,j=1}^{N_\a} (|\by_{\a ij}|-L')_+^2  \leq C( 1 + \cD_\a^{(2-\b)_+} \cP_\a).
	\end{equation}
	Continuing the above,
	\[
	\left| \frac{1}{N_\a} \sum_{k=1}^{N_\a}  \n U(\radial{\by_{\a ik}}) \cdot \bw_{\a k} \right|  
	\leq c_4 \sqrt{\cK_\a} ( 1 +  \cD_\a^{2-\b} \cP_\a)^{1/2}  \leq c_5 \sqrt{\cK_\a} (1+\cE_{\a \infty})^{\frac{2-\b}{2\b}} .
	\]
	In this case,
	\begin{equation}\label{e:parE2}
	\ddt \cE_{\a \infty} \leq c_6 \sqrt{\cK_\a} (1+\cE_{\a\infty})^{\frac{2-\b}{2\b}}  \quad \Rightarrow \quad   \cE_{\a\infty} \lesssim \lan t\ran^{ \frac{2\b}{3\b - 2} }   \quad \Rightarrow \quad  \cD_\a \leq \lan t\ran^{\frac{2}{3\b - 2} }.
	\end{equation}
	
	Finally, for $\b >2$, we argue similarly, using that $|D^2U(\radial{\by_{\a ik}})| \leq \cD_\a^{\b-2}$, and \eqref{e:sqP},
	to obtain
	\[
	\left|  \frac{1}{N_\a} \sum_{k=1}^{N_\a}  \n U(\radial{\by_{\a ik}}) \cdot \bw_{\a k} \right| \leq \sqrt{\cK_\a} \cD_\a^{\b-2} ,
	\]
	and hence,
	\begin{equation}\label{e:parE3}
	\ddt \cE_{\a\infty} \leq c_7 \sqrt{\cK_\a} (1+\cE_{\a\infty})^{\frac{\b-2}{\b}} \quad \Rightarrow \quad   \cE_{\a\infty} \lesssim \lan t\ran^{ \frac{\b}{2} } \quad \Rightarrow \quad \cD_\a \leq \lan t\ran^{\frac{1}{2} }.
	\end{equation}
	We have proved the following a priori estimate:
	\begin{equation}\label{e:DupperU}
	\cD_\a(t) \lesssim \lan t\ran^d, \quad \text{ where }\quad  d =  \left\{
	\begin{split}
	1&, \quad 1\leq \b < \frac43,\\
	\frac{2}{3\b - 2} &, \quad \frac43 \leq \b < 2,\\
	\frac{1}{2} &, \quad \b \geq 2.
	\end{split}\right.
	\end{equation}
	
	Denote $\displaystyle 	\zeta(t) = \lan t\ran^{-\g d}$.
	Then according to the basic energy equation  \eqref{e:enlawU}  we have
	\begin{equation}\label{e:enprecoer}
	\ddt \cE_\a \leq - \frac12 \cI_\a - c \zeta(t) \cK_\a - \e R_\a(t) \cK_\a.
	\end{equation}
	Considering this as a starting point, just like in the quadratic confinement case, we will build correctors to the energy to achieve full coercivity on the right hand side of \eqref{e:enprecoer}.  We introduce one more auxiliary power function
	\[
	\eta(t) =  \lan t\ran^{- a}, \quad  \g d \leq a <1.
	\]
	First, we consider the same longitudinal momentum
	\[
	\cX_\a = \frac{1}{N_\a} \sum_{i=1}^{N_\a} \by_{\a i} \cdot \bw_{\a i}.
	\]
	It will come with a pre-factor $\d \eta(t)$, where $\d$ is a small parameter.  Let us estimate using \eqref{e:sqP}:
	\[
	\d \eta(t) |\cX_\a| \leq  \d \cK_\a + \d \eta^2(t) \frac{1}{N_\a^2} \sum_{i,j=1}^{N_\a} |\by_{\a ij}|^2 \leq  \d \cK_\a + c \d \eta^2(t) +  \d \eta^2(t)  \cD_\a^{(2-\b)_+} \cP_\a. 
	\]
	The potential term is bounded by $\d \cP_\a$ as long as 
	$	2a \geq d(2-\b)_+$.
	Hence, 
	\begin{equation}\label{e:enX}
	\d \eta(t) |\cX_\a| \leq \d \cE_\a + c \eta^2(t).
	\end{equation}
	This shows that  
	\[
	\cE_\a + \d \eta(t) \cX_\a + 2c \eta^2(t) \sim \cE_\a + c \d \eta^2(t). 
	\]
	Let us now consider the derivative
	\begin{multline*}
	\cX_\a'  = \frac{1}{N_\a} \sum_{i=1}^{N_\a} |\bw_{\a i}|^2 +\frac{1}{N_\a^2} \sum_{i,k=1}^{N_\a} \by_{\a ik} \cdot \bw_{\a ki} \phi_{\a ki}  - \frac{1}{N_\a^2} \sum_{i,k=1}^{N_\a}  \by_{\a ik} \cdot \n U(\radial{\by_{\a ik}}) - \e R_\a(t) \cX_\a \\
	 = \cK_\a+A-B - \e R_\a(t) \cX_\a.
	\end{multline*}
	The gain term $B$, by convexity dominates the potential energy $B \geq \cP_\a$.  As to $A$:
	\[
	|A| \leq \frac{|\phi|_\infty}{2 \d^{1/2}\eta(t)} \cI_\a + \frac{\d^{1/2} \eta(t)}{2} \frac{1}{N_\a^2} \sum_{i,j=1}^{N_\a}|\by_{\a ij}|^2 \lesssim  \frac{1}{ \d^{1/2}\eta(t)} \cI + \d^{1/2} \eta(t) +  \d^{1/2} \eta(t) \cD_\a^{(2-\b)_+} \cP_\a.
	\]
	By requiring a more stringent assumption on parameters
	\[%\begin{equation}\label{key}
	\a \geq d(2-\b)_+,
	\]%\end{equation}
	we can ensure that the potential  term is bounded  by $\sim \e^{1/2} \cP$, which can be absorbed by the gain term. 
	
The inter-flock term in \eqref{e:enprecoer} helps abosrb the corresponding residual term $\e R_\a(t) \cX_\a$. Indeed,
\[
\begin{split}
	\e R_\a(t) \cX_\a & \leq  \frac{1}{2 \d \eta(t)} \e R_\a(t) \cK_\a + \e R_\a(t)\d \eta(t)  \sum_{i,j=1}^{N_\a}|\by_{\a ij}|^2 \\
	  & \leq  \frac{1}{2 \d \eta(t)} \e R_\a(t) \cK_\a  + C_1 \d \eta(t) + C_2 \d \eta(t) \cD_\a^{(2-\b)_+} \cP_\a,
\end{split}
\]
with the latter absored into the gain term as in the case of $A$. 
	
So far, we have obtained
	\begin{equation}\label{e:auxE1}
	\ddt ( \cE_\a + \d \eta(t) \cX_\a + 2c \eta^2(t) ) \leq - c_1 \d \eta(t) \cE + c_2 \eta^2(t) + \d \eta'(t) \cX_\a.
	\end{equation}
In view of \eqref{e:enX},
	\[
	| \d \eta'(t) \cX_\a | \leq \d \frac{1}{\lan t \ran} \eta(t) |\cX_\a| \leq  \d \frac{1}{\lan t \ran} \cE_\a +  \d \frac{\eta^2(t)}{\lan t \ran}.
	\]
	Since $a <1$, the energy term will be absorbed, and the free term is even smaller then $\eta^2$.  Denoting 
	\[
	E_\a =  \cE_\a + \d \eta(t) \cX_\a + 2c \eta^2(t) ,
	\]
	we obtain 
	\[
	\ddt E_\a \leq  - c_1 \eta(t) E_\a + c_2 \eta^2(t).
	\]
	By Duhamel's formula,
	\[
	E_\a(t) \lesssim  \exp\{ - \lan t \ran^{1-a}\} +   \exp\{ - \lan t \ran^{1-a}\} \int_0^t   \frac{e^{\lan s \ran^{1-a}}}{ \lan s \ran^{2 a}} \ds.
	\]
	By an elementary asymptotic analysis,
	\[
	\int_0^t   \frac{e^{\lan s \ran^{1-a'}}}{ \lan s \ran^{a''}} \ds \sim  \exp\{  \lan t \ran^{1-a'}\} \frac{1}{\lan t \ran^{ a''-a'}}.
	\]
	Thus, we obtain an algebraic decay rate
	\begin{equation}\label{e:drini}
	E_\a(t) \lesssim \frac{1}{\lan t \ran^{a}}, \quad \forall  a<1,
	\end{equation}
	provided
	\begin{equation}\label{e:dgb}
	d \g <1 \quad \text{ and }\quad   d (2- \b)_+ <1.
	\end{equation}
	This translates exactly into the conditions on $\g$ given by \eqref{e:gb}, and \eqref{e:drini}  automatically implies \eqref{e:driniE}
	
	Going back to the estimates \eqref{e:parE1} and \eqref{e:parE2}, but keeping the kinetic energy with its established decay, we obtain a new decay rate for the diameter
	\[
	\cD_\a \leq C_\d \lan t \ran^{\frac{d}{2} + \d }, \quad \forall \d>0.
	\]
	
	At the next stage we prove flocking: $\cD_\a(t)< \bar{D}_\a$. In order to achieve this we return again to the particle energy estimates. Let us denote
	\[
	\cP_{\a i} =  \frac{1}{N_\a} \sum_{k=1}^{N_\a}  U(\radial{\by_{\a ik}}), \quad \cI_{\a i} =  \frac{1}{N_\a} \sum_{k=1}^{N_\a}  \phi_{\a ik}| \bw_{\a ki}|^2, \quad \cX_{\a i} = \by_{\a i} \cdot \bw_{\a i}.
	\]
	Using \eqref{e:Ei1}, \eqref{e:Ei2}, \eqref{e:potS}, \eqref{e:sqP} and the fact that $\cD_\a^{(2-\b)_+} \cP$ has a negative rate of decrease, we obtain
	\begin{multline*}
			\ddt \cE_{\a i} \leq  \cK_\a  - \frac12 \phi_\a(\cD_\a) |\bw_{\a i}|^2  - \cI_{\a i} + c \sqrt{\cK_\a} - \e R_\a(t) |\bw_{\a i}|^2  \\
			\lesssim  - \frac12 \phi_\a(\cD_\a) |\bw_{\a i}|^2  - \cI_{\a i} + \frac{1}{\lan t \ran^{\frac12 - \d}} - \e R_\a(t) |\bw_{\a i}|^2, \quad \forall \d>0.
	\end{multline*}
		In view of \eqref{e:dgb}, we can pick $a$ and small $b$ such that 
	\begin{equation}\label{e:adsub}
	\begin{split}
	\frac{d \g}{2} + b \g  < \frac12 - 2 b &< a < \frac12 -  b \\
	(2-\b)_+ d + 2\d (2-\b)_+ & < 2 a.
	\end{split}
	\end{equation}
	We use as before the auxiliary rate function $\eta(t) =  \lan t\ran^{- a}$.  Let us estimate the corrector
	\[
	| \d \eta(t) \cX_{\a i} | \leq \e |\bw_{\a i}|^2 + \d \eta^2(t) |\by_{\a i}|^2 \leq \d |\bw_{\a i}|^2 + \d  \eta^2(t) \cD_\a^{2-\b} \cP_{\a i} + L^2 \d  \eta^2(t) \leq  \d |\bw_{\a i}|^2 + c \d \cP_{\a i} + L^2 \d  \eta^2(t).
	\]
	So, 
	\[
	E_{\a i} : = \cE_{\a i} + \d \eta(t) \cX_{\a i} + 2 L^2 \d  \eta^2(t) \sim \cE_{\a i}  + L^2 \d  \eta^2(t).
	\]
	Differentiating,
	\[
	\begin{split}
	\cX'_{\a i} & = |\bw_{\a i}|^2 + \frac{1}{N_\a}\sum_{k =1}^{N_\a} \by_{\a i} \cdot \bw_{\a ki} \phi_{\a ki}  -  \frac{1}{N_\a}\sum_{k =1}^{N_\a}  \by_{\a ik} \cdot \n U(\radial{\by_{\a ik}}) \\
	& +  \frac{1}{N_\a}\sum_{k =1}^{N_\a}  \by_{\a k} \cdot ( \n U(\radial{\by_{\a ik}}) -  \n U(\radial{\by_{\a i}})) -  \e R_\a(t) \cX_{\a i} \\
	& \leq   |\bw_{\a i}|^2  + \d^{1/2} \eta(t) |\by_{\a i}|^2 + \frac{1}{\d^{1/2} \eta(t)} \cI_{\a i}  - \cP_{\a i} +  \frac{1}{N_\a^2} \sum_{l,k=1}^{N_\a}  |\by_{\a kl}|^2 \\
	&+ \frac{1}{2 \d \eta(t)} \e R_\a(t)  |\bw_{\a i}|^2 + 2 \d \eta(t) \e R_\a(t) |\by_{\a i}|^2 \\
	\intertext{where the last term is already smaller than $\d^{1/2} \eta(t) |\by_{\a i}|^2$ for small enough $\d$,}
	& \leq  |\bw_{\a i}|^2  +  \d^{1/2} L^2 \eta(t)  +  \d^{1/2} \cD_\a^{(2-\b)_+} \eta(t) \cP_{\a i} +  \frac{1}{\d^{1/2}\eta(t)} \cI_{\a i}  - \cP_{\a i} + C + \frac{1}{2 \d \eta(t)} \e R_\a(t)  |\bw_{\a i}|^2\\
	\intertext{in view of \eqref{e:adsub}, $\e^{1/2} \cD^{(2-\b)_+} \eta(t)  \lesssim \e^{1/2}$, so the potential term is absorbed by $- \cP_i$,}
	& \leq  |\bw_{\a i}|^2 + \frac{1}{\eta(t)} \cI_{\a i} - \frac12 \cP_{\a i} + C + \frac{1}{2 \d \eta(t)} \e R_\a(t)  |\bw_{\a i}|^2.
	\end{split}
	\]
	Again in view of \eqref{e:adsub}, $\eta(t)$ decays faster than $ \phi_\a (\cD_\a)$, so plugging into the energy equation we obtain
	\[
	\ddt E_{\a i} \leq  - \d \eta(t) E_{\a i} +   \eta(t) +  \sqrt{\cK_\a}  + \d \eta'(t) \cX_{\a i},
	\]
	and as before $\d \eta'(t) \cX_{\a i}$ is a lower order term which is absorbed into the negative energy term and $+ \eta^2$. So,
	\[
	\ddt E_{\a i} \leq  - \d \eta(t) E_{\a i} +   \eta(t) +  \sqrt{\cK_\a}.
	\]
	By our choice of constants \eqref{e:adsub}, $\sqrt{\cK_\a}$ decays faster than $\eta(t)$, hence, 
	\[
	\ddt E_{\a i} \lesssim  - \d \eta(t) E_{\a i} +   \eta(t).
	\]
	This proves boundedness of $E_{\a i}$,  and hence that  of $\cE_{\a i}  + L^2 \d  \eta^2(t)$, and hence that of $\cE_{\a i}$. In view of \eqref{e:diamE}, this implies the flocking bound $	\cD_\a(t) < \overline{\cD}_\a$ for all $t>0$.
\end{proof}

It is interesting to note that when the support of the potential spans the entire line, $L =0$, and $U$ lands at the origin with at least a quadratic touch:
\begin{equation}\label{e:quadland}
U(r) \geq a_0 r^2, r<L',
\end{equation}
then we can establish exponential alignment in terms of the energy $\cE_\a$. Indeed, since we already know that the diameter is bounded, the basic energy equation reads
\[
\ddt \cE_\a \leq - c_0 \cK_\a - \frac12 \cI_\a.
\]
The momentum corrector needs only an $\d$-prefactor to satisfy the bound
\[
| \d \cX_\a | \leq \d \cK_\a + \d  c \cP_\a.
\]
This is due to the assumed quadratic order of the potential near the origin and, again, boundedness of the diameter.  Hence, $\cE_\a + \d \cX_\a \sim \cE_\a$. The rest of the argument is similar to the general case. We obtain
\[
\cX_\a \lesssim \cK_\a + \d^{1/2} \cP_\a + \frac{1}{\d^{1/2}} \cI_\a  - \cP_\a \leq  \cK_\a  - \frac12 \cP_\a \frac{1}{\d^{1/2}} \cI_\a 
\]
Thus,
\[
\ddt(  \cE_\a  + \d \cX_\a ) \leq - c_1 \cE_\a \sim - c_1 (  \cE_\a + \d \cX_\a ) .
\]
This proves exponential decay of the energy $\cE_\a $.  Going further to consider the individual particle energies, we discover similar decays. Indeed, denoting by  $\Exp$ any quantity that decays exponentially fast, we follow the same scheme:
\[
\ddt \cE_{\a i} \leq - c_1 |\bw_{\a i}|^2 - \frac12 \cI_{\a i}  + \Exp, \qquad \Exp\lesssim e^{-Ct}.
\]
In view of $|\by_{\a i}|^2 \lesssim \cP_{\a i}$, 
\[
\d |\cX_{\a i}| \leq  \d |\bw_{\a i}|^2 + \d \cP_{\a i},
\]
so $\cE_{\a i} + \d \cX_{\a i} \sim \cE_{\a i}$.  Further following the estimates as in the proof, 
\[
\cX_{\a i}' \lesssim  |\bw_{\a i}|^2  + \frac{1}{\d^{1/2}} \cI_{\a i} - \frac12 \cP_{\a i}.
\]
Thus,
\[
\ddt ( \cE_{\a i} + \d \cX_{\a i} ) \leq - c_1 ( \cE_{\a i} + \d \cX_{\a i} ) + \Exp.
\]
This establishes exponential decay for $\cE_{\a \infty}$, and hence for the individual velocities. This also proves that $\cD_\a(t) = \Exp$. So, the long time behavior here is characterized by exponential aggregation to a point. 

\begin{theorem}\label{t:aggr}
	Let us assume that the support of the potential spans the entire space and \eqref{e:quadland}.  Then the solutions aggregate exponentially fast:
	\[
	\cD_\a(t) + \max_{i}|\bv_{\a i}(t) - \bV_\a(t) |_\infty \leq C e^{-\d t},
	\]
	for some $C,\d>0$. 
\end{theorem}

\section{Hydrodynamics of multi-flocks}\label{sec:hydro}

In the case of smooth communication kernels, one can formally derive the corresponding kinetic model from \eqref{e:CS} via the BBGKY hierarchy. Let $f_\a(x,v,t)$ denote a density distribution of the $\a$-flock, and define the corresponding flock parameters :
\begin{equation}\label{e:macrohydro}
\begin{split}
M_\a  &= \int_{\R^{2d} } f_\a(\bx,\bv,t) \dbx \dbv, \quad \bX_\a =  \frac{1}{M_\a} \int_{\R^{2d}} \bx f_\a(\bx,\bv,t) \dbx \dbv,\\
\bV_\a  &= \frac{1}{M_\a} \int_{\R^{2d}} \bv f_\a(\bx,\bv,t) \dbx \dbv.
\end{split}
\end{equation}
The kinetic model reads as follows
\begin{equation}\label{e:kin}
\p_t f_\a + \bv \cdot \n_{\bx} f_\a + \l \n_{\bv}\cdot \bQ_\a (f_\a,f_\a) + \e \n_\bv \cdot \left[ \sum_{\b \neq \a} M_\b \psi(\bX_\a,\bX_\b)  (\bV_\b - \bv ) f_\a \right] = 0,
\end{equation}
where 
\begin{equation}\label{eq:STQ}
\bQ_\a (f,f)(\bx,\bv,t) = f(\bx,\bv,t) \int_{\R^{2d}}
\phi_\a(\bx,\bx') (\bv'-\bv) f(\bx',\bv,t') \dbx'\dbv' .
\end{equation}
The macroscopic system can be obtained, again formally, from \eqref{e:kin} by considering monokinetic closure $f_\a = \d_0(\bv-\bu_\a(\bx,t)) \rho_\a(\bx,t)$. The resulting system presents as hybrid of hydrodynamic and discrete parts, where the hydrodynamic part corresponds to the classical CS dynamics within flocks, while the discrete part governs communication of a given flock with other flocks' averaged quantities.  To write down the equations, we denote macroscopic variables by $(\rho_\a,\bu_\a)_{\a=1}^A$,
\[
\rho_\a(\bx,t) = \int_{\R^{d} } f_\a(\bx,\bv,t) \dbv, \quad \rho_\a \bu_\a = \int_{\R^{d} } \bv f_\a(\bx,\bv,t) \dbv,
\]
while \eqref{e:macrohydro} represent upscale parameters of the flocks.  The full hydrodynamic system  reads
\begin{equation}\label{e:CShydro}
\left\{
\begin{split}
\p_t \rho_\a + \n \cdot (\bu_\a \rho_\a) & = 0\\
\p_t {\bu}_{\a} + \bu_\a \cdot \n \bu_\a & =\l_\a \int_{\R^d} \phi_\a(\bx,\by)(\bu_\a(\by) - \bu_\a(\bx)) \rho_\a(\by) \dby  \\
&+ \e \sum_{\b \neq \a} M_\b \psi(\bX_\a ,\bX_\b)[ \bV_\b - \bu_\a(\bx,t) ].
\end{split}\right.  \qquad \a = 1,\ldots,A.
\end{equation}
Writing the momentum equation in conservative form we obtain
\begin{equation}\label{eq:moment}
\begin{split}
\p_t ( \rho_\a \bu_{\a}) + \n_x (\rho_\a \bu_\a \otimes \bu_\a )  &=\l_\a \int_{\R^d} \phi_\a(\bx,\by)(\bu_\a(\by) - \bu_\a(\bx)) \rho_\a(\bx) \rho_\a(\by) \dby \\
& \ \ \ +\mu \sum_{\b \neq \a} M_\b \psi(\bX_\a ,\bX_\b)[ \bV_\b - \bu_\a(\bx,t) ]\rho_\a(\bx).
\end{split}
\end{equation}
Integrating \eqref{eq:moment} over $\R^d$, system \eqref{e:CShydro} upscales to the same  discrete Cucker-Smale system \eqref{e:CSmacro} for  macroscopic parameters $\{ \bX_\a, \bV_\a\}_\a$.

%%%%%%%%%%%%%%%%%%%%%%%%%%%%%%%
\subsection{Slow and fast alignment of hydrodynamic multi-flocks}
%%%%%%%%%%%%%%%%%%%%%%%%%%%%%%%%%%%
As in the discrete case, we will deal with kernels that admit fat tail subkernels \eqref{e:subk}. Alignment dynamics for hydrodynamic description mimics that of the discrete one once we pass to Lagrangian coordinates. 
Denote  by $\bx_\a(\bx,t)$ the characteristic flow map of the $\bu_\a$. From the continuity equation we conclude that the mass measure $\rho_\a(\by,t) \dby$ is the push-forward of the initial measure $\rho_\a(\by,0) \dby$ by the flow. So, passing to the Lagrangian coordinates $\bv_\a(\bx,t) = \bu_\a(\bx_\a(\bx,t),t)$ we obtain
\begin{multline*}
\ddt {\bv}_{\a}=\l_\a \int_{\R^d} \phi_\a(\bx_\a(\bx,t),\bx_\a(\by,t))(\bv_\a(\by) - \bv_\a(\bx)) \rho_\a(\by,0) \dby \\ + \mu \sum_{\b \neq \a} M_\b \psi(\bX_\a ,\bX_\b)[ \bV_\b - \bv_\a(\bx,t) ].
\end{multline*}
Passing to the reference frame moving with the average velocity in each flock:
\begin{equation}
\begin{split}
\bw_\a(\bx,t) := \bv_\a(\bx,t) - \bV_\a(t)
\end{split}
\end{equation}
we obtain the momentum system quite similar to its discrete counterpart \eqref{e:CSshift}
\[
\ddt \bw_\a(\bx,t) = \l_\a \int_{\R^d} \phi_\a(\bx_\a(\bx,t),\bx_\a(\by,t))(\bw_\a(\by,t) - \bw_\a(\bx,t)) \rho_\a(\by,0) \dby   - \e R_\a(t) \bw_\a.
\]
Thus, all the alignment statements of \thm{t:fast} and \thm{t:slow} carry over directly to these settings. In the original variables these translate into the following.  
\begin{theorem}\label{t:hydroalign}
Assuming that the initial diameter of the $\a$-flock is finite, and $\phi_\a$ has fat tail, the $\a$-flock aligns at a rate dependent on $\l_\a$ :
\[
  \diam \left( \supp \rho_\a(\cdot,t) \right) < \overline{\cD}_\a, \quad  \max_{\bx\in \supp \rho_\a(\cdot,t)} |\bu_\a(\bx,t) - \bV_\a(t)| \lesssim e^{-\d_\a t},
   \]
where $ \d_\a = \l_\a M_\a \phi_\a(\overline{\cD}_\a)$. 
Furthermore, if $\psi$ has a fat tail, the kernels $\phi_\a \geq 0$ are arbitrary, and the multi-flock has a finite diameter initially, then global alignment occurs at a rate dependent on $\mu$:
\[
\diam \left( \cup_\a \supp \rho_\a(\cdot,t) \right) < \overline{\cD}, \quad  \max_{\bx\in \supp \rho_\a(\cdot,t), \a = 1,...,A} |\bu_\a(\bx,t) - \bV| \lesssim e^{-\d t },
 \]
 where $\d = \mu M \psi(\overline{\cD})$.
\end{theorem}

\subsection{External forcing}
Theorems \ref{t:attr} and \ref{t:aggr}  have similar analogues for the system with additional  external interaction forces \cite{ST2019a}
\[
\bF_\a =  - \n_\bx U \ast \rho_\a.
\]
This is due to the fact that our arguments establish rates independent of the number of agents. The hydrodynamic proofs repeat the discrete case ad verbatim, we therefore leave them out entirely.

%%%%%%%%%%%%%%%%%%%%%%%%%%%%%%%%%%%%
\subsection{Global existence and 1D multi-flocking: smooth kernel case}\label{sec:1Dsmooth}
%%%%%%%%%%%%%%%%%%%%%%%%%%%%%%%%%%%
We restrict attention to radial communication kernels $\phi_\a, \psi \in W^{2,\infty}$. The most convenient form of \eqref{e:CShydro} to study regularity is in the shifted reference frame attached to the flock:
\[
v_\a(x,t) := u_\a(x - X_\a(t),t) - V_\a(t), \qquad r_\a := \rho_\a(x - X_\a(t),t).
\]
The new pair satisfies
\begin{equation}\label{e:CShydro1D}
\left\{
\begin{split}
\p_t r_\a + (v_\a r_\a)' & = 0\\
\p_t {v}_{\a} + v_\a v_\a' & =\l_\a \int_{\R^d} \phi_\a(\radial{x-y})(v_\a(y) - v_\a(x)) r_\a(y) \dy   - \e R_\a(t) v_\a ,\\
R_\a(t) & = \sum_{\b \neq \a} M_\b \psi(\radial{X_\a -X_\b}).
\end{split}\right. 
\end{equation}

    In the case of the classical hydrodynamic alignment system the global well-posedness in 1D  relies on a threshold condition for the auxiliary quantity $e = v' + \phi \ast \rho$, which satisfies the same continuity as the density, see \cite{TT2014}. For multi-flocks we define, accordingly,  the family of such quantities 
\[
e_\a(x,t) = v'_\a + \l_\a \phi_\a \ast r_\a.
\]
By virtue of \eqref{e:CShydro1D}, $e_\a$ satisfies
\[
\p_t e_\a + (v_\a e_\a)' = -\e R_\a(t) v'_\a,
\]
which can be written as a non-autonomous logistic equation along characteristics of $v_\a$:
\begin{equation}\label{e:log}
	\ddt {e}_\a =  (\e R_\a + e_\a)( \phi_\a \ast r_\a-e_\a ), \qquad \ddt:=\p_t +v_\a\p_x.
\end{equation}
It is therefore natural to a expect threshold condition to guarantee global existence. We elaborate on that in the next result.  
\begin{theorem}[Global existence]\label{t:hydro1D}
Let $\psi, \phi_\a \in W^{2,\infty}(\R)$. For any initial conditions $(u_\a,\rho_\a) \in W^{2,\infty}  \times (W^{1,\infty} \cap L^1) $ satisfying
	\begin{equation}\label{e:thres}
 u'_\a(x,0) + \l_\a \phi_\a \ast \rho_\a(x,0) \geq 0 \quad \mbox{ for all } \ x\in \R, \ \a = 1,\ldots,A
	\end{equation}
there exists a unique global solution  $(u_\a,\rho_\a)\in L^\infty_\loc([0,\infty); W^{2,\infty}  \times (W^{1,\infty}\cap L^1))$.  On the other hand, if for some $x_0\in \R$ and $\a\in \{1,...,A\}$
\begin{equation}\label{e:threshill}
u'_\a(x_0,0) + \l_\a \phi_\a \ast \rho_\a(x_0,0) <  - \e M \psi(0),
\end{equation}
then the solution develops a finite time blowup.
\end{theorem}
The gap between the threshold levels is due to the fact that it is hard to predict the dumping coefficient $\e R_\a(t)$, which may fluctuate in time. In particular, if $\psi$  has a fat tail, then the argument below shows that the threshold for global existence is improvable to 
		\begin{equation}\label{e:thresD}
	e_\a(x,0) \geq - \e  M \psi(\overline{\cD}) \quad \mbox{ for all } x\in \R, \ \a = 1,\ldots,A,
	\end{equation}
where $\bar{D}$ is determined from the initial conditions by equation \eqref{e:Dinfty}:
	\begin{equation}\label{e:Dinfty}
\e \int_{\cD_0}^{ \overline{\cD} } \psi (r) \dr =  \cA_0.
\end{equation}

\begin{proof} Let us start with the negative result.   Noting that $ \e M \psi(0)$ is the global upper  bound  for $\e R_\a$, from \eqref{e:log} we conclude that $\ddt e_\a \leq 0$. So, $e_\a$ will remain below $-\e (1+\d) M \psi(0)$ for some $\d>0$ along the characteristics starting at $x_0$. Hence,
\[
\ddt {e}_\a \leq  \frac{\d}{1+\d} e_\a ( \phi_\a \ast r_\a-e_\a )\lesssim -e_\a^2.
\]
Hence, $e_\a$ blows up in finite time.

On the other hand, if \eqref{e:thres} holds initially, then since 
\[
 e_\a( \phi_\a \ast r_\a-e_\a ) \leq	\dot{e}_\a \leq  (\e R_\a + e_\a)(|\phi_\a|_\infty M_\a -e_\a ),
\]
$e_\a$ will remain non-negative and asymptotically bounded from above by $|\phi_\a|_\infty M_\a$. Hence, $\|v_\a'\|_\infty$ is uniformly bounded.  Next, solving the continuity equation along characteristics
\[
r_\a(x_\a(x_0;t),t) = r_\a(x_0,0) \exp\left\{- \int_0^t v_\a'(x_\a(x_0;s),s) \ds \right\},
\]
we conclude that $r_\a$ remains a priori bounded on any finite time interval.  

Next, differentiating the $e$-equation,
\[
\ddt e_\a' + v_\a''e_\a + 2v'_\a e'_\a + v_\a e''_\a = - \e R_\a v_\a'',
\]
passing to Lagrangian coordinates and replacing $v_\a'' = e_\a' - \l_\a \phi_\a' \ast \rho_\a$ we obtain, in view of already known information,
\[
 \ddt |e_\a' |^2 \leq f(t) |e_\a'|^2 + g(t),
\]
where $f$ and $g$ are  bounded functions. Hence, $e'_\a$ remain bounded as well, and consequently so does $v''_\a$.  Finally, $r'_\a \in L^\infty$ follows from differentiating and integrating the continuity equation.

The obtained a priori estimates lead to construction of global solutions by the standard approximation and continuation argument.

\end{proof}

We proceed with two strong flocking results that demonstrate alignment in cluster with inter-flock slow and inner-flock fast rates as expected.

\begin{theorem}[{\bf Strong flocking}]\label{t:faststrong}
	Suppose the threshold condition \eqref{e:thres} holds so the solution exists globally. If for some $\a \in \{1,\ldots,A\}$ the $\a$-flock has compact support and the kernel $\phi_\a$ has a fat tail,  then there exists $ \d_\a = \d_\a(\phi_\a, \l_\a ,u_\a(0),\rho_\a(0))$ such that 
	\[
	\sup_{x \in \supp \rho_\a(\cdot,t)}|u_\a (x,t) - U_\a(t)| + |u_\a'(x,t)|+ |u_\a''(x,t)|  \lesssim  e^{-\d_\a t},
	\]
and the density $\rho_\a$ converges to a traveling wave with profile $\bar{\rho}_\a$ in the metric of $C^\g$ for any $0<\g<1$:
\[	
		\|\rho_\a(\cdot,t) - \bar{\rho}_\a(\cdot - X_\a(t)) \|_{C^\gamma} \lesssim e^{-\d_\a t}.
\]

Furthermore, if $\psi$ has a fat tail, the kernels $\phi_\a \geq 0$ are arbitrary, and the multi-flock has a finite diameter initially, then global alignment occurs at a rate $\d = \d(\psi, \e,u(0),\rho(0))$:
\[
\begin{split}
\sup_{x \in \supp \rho_\a(\cdot,t), \a = 1,...,A}|u_\a (x,t) - U| + |u_\a'(x,t)|+ |u_\a''(x,t)|  &\lesssim e^{-\d t}, \\
\|\rho_\a(\cdot,t) - \bar{\rho}_\a(\cdot - U t) \|_{C^\gamma} & \lesssim e^{-\d t}.
\end{split}
\]
\end{theorem}

\begin{proof} Let us prove the local statement first. Note that the alignment itself is a consequence of \thm{t:hydroalign}. Plus we have a global bound $\overline{\cD}_\a$ on the diameter of the $\a$-flock.   Next, let us make the following observation: since 
	\[
	\phi_\a \ast \rho_\a(x) \geq M_\a \phi_\a(\overline{\cD}_\a) = c_0, \quad \forall x\in \supp r_\a,
\]
then from \eqref{e:log} we obtain
\[
\ddt {e}_\a \geq  e_\a( c_0 -e_\a ).
\]
Consequently, there exists a time $t_0$ starting from which $e_\a(x) \geq c_0/2$ for all $x\in \supp r_\a$. This follows by direct solution of the ODI.

Let us now write the equation for $v_\a'$
\begin{equation}\label{e:vder}
\ddt v_\a' + v_\a v_\a'' = \int_\R \phi_\a'(x-y)(v_\a(y) - v_\a(x)) r_\a(y) \dy  - (\e R_\a(t) + e_\a) v_\a'.
\end{equation}
We already know from \thm{t:hydroalign} that the velocity variations are exponentially decaying with the desired rate. Let us denote by $E(t)$ a generic function with such exponential decay. Then, in Lagrangian coordinates,
\[
\ddt |v_\a'|^2  \leq  E(t) v_\a' - \frac{c_0}{2} |v_\a'|^2 \leq E(t) -  \frac{c_0}{4} |v_\a'|^2.
\]
This establishes the decay for $v_\a'$ on the support of $r_\a$.  Next, 
\begin{equation}
\begin{split}
\ddt v_\a''  + 2v_\a' v_\a'' & = \int_\R \phi_\a''(x-y)(v_\a(y) - v_\a(x)) r_\a(y) \dy \\
&-2 v_\a' \phi_\a'\ast r_\a  -  (\e R_\a(t) + e_\a) v_\a''.
\end{split}
\end{equation}
So, similar to the previous
\[
\ddt |v_\a''|^2  \leq  E(t) -  \frac{c_0}{4} |v_\a''|^2.
\]
Thus, $|v_\a''| \sim E(t)$.  As to the density,
\begin{equation}\label{e:densE1}
\ddt r_\a'  = - 2 v_\a' r_\a' - v_\a'' r_\a = E(t) r_\a' + E(t),
\end{equation}
and we obtain uniform in time control over $\|r_\a'\|_\infty$. 

To conclude strong flocking we write
\begin{equation}\label{e:densE0}
\ddt r_\a = -v_\a r_\a' - v_\a' r_\a = E(t).
\end{equation}
This shows that $r_\a(t)$ is Cauchy in $t$ in the metric of $L^\infty$. Hence, there exists $\bar{r}_\a \in L^\infty$ such that $\| r_\a(t) - \bar{r}_\a\|_\infty = E(t)$. Since $r_\a'$ is uniformly bounded, this also shows that $\bar{r}_\a$ is Lipschitz. Convergence in $C^\g$, $\g<1$, follows by interpolation. Finally, passing to the original coordinate frame gives the desired result.

As to the global statement, the result follows from exact same argument above by noting that $\e R_\a(t) \geq \e M \psi(\bar{D}) = c_0$, and all the macroscopic momenta $U_\a$ align by \thm{t:hydroalign}. 
\end{proof}

\begin{remark} We note that the strong flocking result is new even in the classical mono-flock context. The work \cite{ST2} treats the more restrictive case of a kernel with positive infimum, while \cite{TT2014} only claims bounded diameter. 
\end{remark}

%%%%%%%%%%%%%%%%%%%%%%%%%%%%%%%%%%%
\subsection{Global existence and 1D multi-flocking: singular kernel case}\label{sec:1Dsing} 
%%%%%%%%%%%%%%%%%%%%%%%%%%%%%%%%%%%
In the case when $\psi$ is smooth and inner communication kernels are singular
\begin{equation}\label{e:kersing}
	\phi_\a(r) = \frac{1}{r^{1+s}}, \quad 0<s<2,
\end{equation}
the system \eqref{e:CShydro1D} becomes of fractional parabolic type with bounded drift (due to the maximum principle) and bounded dumbing term. Considered under periodic settings $\T$ with no-vaccum initial condition $\rho_\a > 0$, $\forall \a = 1,\ldots, A$, we encounter no additional issues in the application of the regularity results obtained in \cite{ST1,ST2,ST3}. Indeed, the dumping term $\e R_\e \bv_\a$ has no effect on the continuity equation written in parabolic form
\[
\p_t r_\a + v_\a r_\a' + e_\a r_\a = r_\a \L_s r_\a,
\]
where $\L_s=-(-\Delta)^{s/2}$ is the fractional $s$-Laplacian.  As to the momentum equation it can be viewed as a bounded force for the initial H\"older regularization applied from \cite{S2012,SS2016} in the way identical to our previous works. Further adaptation of the non-local maximum principal estimates of Constantin-Vicol \cite{CV2012} and continuation criteria for higher order Sobolev spaces is straightforward. 
\begin{theorem}
	Let $\psi$ is a smooth kernel and $\phi_\a$ be the kernel of $\L_s$ on $\T^1$. Then system \eqref{e:CShydro} admits a global solition for any initial data in $u_\a \in H^4(\T^1)$, $\rho_\a \in H^{3+s}(\T^1)$ with no vacuum:
	\[
	\min_{\a,x\in\T^1} \rho_\a(x,0) > 0.
	\]
	The solution belongs locally to 
	\[
	u_\a \in C([0,\infty), H^4) \cap L^2([0,\infty), H^{4+ \frac{s}{2}}), \quad 	\rho_\a \in C([0,\infty), H^{3+s}) \cap L^2([0,\infty), H^{3+ \frac{3s}{2}}).
	\]
\end{theorem}

\section{Appendix. Global existence for singular kernels}
 Although collisions between the agents are possible with smooth kernels, this does not cause issues from the point of view of proving global existence of \eqref{e:CS}, using Picard iteration  and continuation.  If the kernels $\phi_\a$ are singular, however,  collisions lead to finite time blowup, so this case needs to be addressed separately.  As was shown in \cite{CCMP2017},  if the kernel is sufficiently singular collisions are prevented by strong close range alignment. We  revisit this result in the context of multi-flocks. 

\begin{theorem}[{\bf Singular communication kernels}]\label{t:coll}
Suppose the $\a$-flock is governed by a singular communication so that
\begin{equation}\label{e:fathead}
\int_0^1 \phi_\a(r) \dr = \infty.
\end{equation}
Then the flock experiences no internal collisions between agents.
\end{theorem}

\begin{proof} The proof given below is a simplified version of the argument given in \cite{CCMP2017}. 
First, we assume for notational simplicity that all the masses are unity.
Let us assume that for a given non-collisional initial condition $(\bx_{\a i}, \bv_{\a i} )_{i,\a}$ a collision occurs at time $T^*$ for the first time. Let $\O^*_\a \ss \O_\a= \{1,\ldots,N_\a\}$ are the indexes of the  agents that collided at one point. Hence, there exists a $\d>0$ such that $|\bx_{\a i}(t) - \bx_{\a k}(t)| \geq \d$ for all $i \in \O^*_\a$ and $k \in \O_\a \backslash   \O^*_\a$. Denote 
\[
\cD^*_\a (t) = \max_{i,j \in \O^*_\a} |\bx_{\a i}(t) - \bx_{\a j}(t)|, \quad \cA_\a^*(t) = \max_{i,j \in \O^*_\a} |\bv_{\a i}(t) - \bv_{\a j}(t)|  =  \max_{\substack{\boldsymbol{\ell} \in \R^n: |\boldsymbol{\ell}| = 1\\ i, j\in \O^*_\a}}  \lan \boldsymbol{\ell}, \bv_{\a i}-\bv_{\a j}\ran.
\]
Directly from the characteristic equation we obtain $| \dot{\cD}^*_\a | \leq \cA_\a^*$, and hence
\begin{equation}
	- \dot{\cD}^*_\a \leq \cA_\a^*.
\end{equation}
Let us fix a maximizing triple $(\ell, i,j)$ for $ \cA_\a^*(t)$ and compute using  the momentum equation 
\[
\begin{split}
\ddt \cA_\a^* &= \sum_{k=1}^N m_{\a k}[ \phi_\a(\radial{\bx_{ik}}) \ell( \bv_{\a k i})  - \phi_\a(\radial{\bx_{jk}}) \ell( \bv_{\a k j})]   - \cA_\a^* R_\a(t) \\
&= \sum_{k\in \O_\a^*} m_{\a k}[ \phi_\a(\radial{\bx_{ik}}) \ell( \bv_{\a k j} -  \bv_{\a i j} )  + \phi_\a(\radial{\bx_{jk}}) \ell( - \bv_{\a k i } - \bv_{\a ij })] \\
&+\sum_{k\not \in \O^*_\a} m_{\a k}[ \phi_\a(\radial{\bx_{ik}}) \ell( \bv_{\a k i})  - \phi_\a(\radial{\bx_{jk}}) \ell( \bv_{\a k j})]     - \cA_\a^* R_\a(t).
\end{split}
\]
The term $ - \cA_\a^* R_\a(t)$ is negative and will be dropped. In the first sum all terms are negative, so we can pull out the minimal value of the kernel which is $\phi_\a(\cD^*_\a)$. In the second sum, all the distances $|\bx_{ik}|$, $|\bx_{jk}|$ are separated by $\d$ up to the critical time $T^*$. So, the kernel will remain bounded. Putting together these remarks we obtain
\[
\ddt \cA_\a^*  \leq C_1 - C_2 \phi_\a(\cD^*_\a)\cA_\a^* .
\]
Let us consider the energy functional
\[
E_\a(t) =  \cA_\a^*(t)  + C_2 \int_{\cD_\a^*(t)}^1 \phi_\a(r) \dr.
\]
From the above we obtain $\ddt E_\a(t)  \leq C_1$.  So, $E_\a$ remains bounded up to the critical time, which implies that $\cD_\a^*(t)$ stays away from zero.
\end{proof}

\begin{corollary}
	Suppose $\psi$ is a smooth kernel, and  each kernel $\phi_\a$ is either smooth or condition \eqref{e:fathead}  holds. Then the system \eqref{e:CS} admits a unique global solution from any initial datum.
\end{corollary}

\noindent
We conclude by noting that this does not preclude collisions between agents from different flocks.\newline

%\bibliographystyle{plain}
%\bibliography{collective-pure,fractional}

\end{document}